\documentclass{article}
\usepackage[utf8]{inputenc}

\usepackage{blkarray}

\usepackage[normalem]{ulem}
\usepackage{cite}

\usepackage[english]{babel} 
\usepackage{graphicx}
\usepackage{float}

\usepackage{authblk}

\newmuskip\pFqskip
\mathchardef\pFcomma=\mathcode`, 

\usepackage{mathtools}

\DeclarePairedDelimiterX\phys[3]{\langle}{\rangle}{#1 \delimsize\vert\mathopen{} #2 \delimsize\vert\mathopen{} #3}

\usepackage{caption}
\usepackage{subcaption}

\usepackage{bm,amsmath,amssymb,comment,esint,accents,physics}
\usepackage{amsthm}
\usepackage{dsfont,enumerate}

\usepackage{tikz}

\usepackage{soul}
\usepackage{color}

\usepackage[margin=0.75in]{geometry}

\makeatletter
\newsavebox{\@brx}
\newcommand{\llangle}[1][]{\savebox{\@brx}{\(\m@th{#1\langle}\)}%
  \mathopen{\copy\@brx\kern-0.5\wd\@brx\usebox{\@brx}}}
\newcommand{\rrangle}[1][]{\savebox{\@brx}{\(\m@th{#1\rangle}\)}%
  \mathclose{\copy\@brx\kern-0.5\wd\@brx\usebox{\@brx}}}
\makeatother

\usepackage[colorlinks,bookmarks=false,citecolor=red,linkcolor=blue,hyperfootnotes=true,urlcolor=magenta]{hyperref}

\newtheorem{theorem}{Theorem}[section]
\newtheorem{proposition}[theorem]{Proposition}
\newtheorem{corollary}{Corollary}[theorem]
\newtheorem{lemma}[theorem]{Lemma}
\theoremstyle{definition}
\newtheorem{definition}[theorem]{Definition}
\newtheorem{remark}[theorem]{Remark}

\newcommand{\al}{\alpha}
\newcommand{\be}{\beta}

\newcommand{\cD}{\mathcal{D}}

\newcommand{\bA}{\boldsymbol{A}}
\newcommand{\x}{\boldsymbol{x}}

\newcommand{\cZ}{\mathcal{Z}}

\newcommand{\NN}{\mathbb{N}}

\newcommand{\po}{\preceq}
\newcommand{\poo}{\preceq_{(\alpha,\beta)}}
\newcommand{\cDext}{\cD^{\text{\rm ext}}}
\newcommand{\lspan}{\text{\rm span}}

\begin{document}

\title{\bf $m$-distance-regular graphs and their relation to multivariate $P$-polynomial association schemes}

\author[1]{Pierre-Antoine Bernard}
\author[2]{Nicolas Crampé}
\author[1,3]{Luc Vinet}
\author[1]{Meri Zaimi}
\author[1]{Xiaohong Zhang}

\affil[1]{\it Centre de Recherches Math\'ematiques (CRM), Universit\'e de Montr\'eal, P.O. Box 6128, Centre-ville
Station, Montr\'eal (Qu\'ebec), H3C 3J7, Canada,}

\affil[2]{\it Institut Denis-Poisson CNRS/UMR7013 - Université de Tours - Université d’Orléans,  Parc de Grandmont, 37200 Tours, France,}

\affil[3]{\it IVADO,  Montréal (Québec), H2S 3H1, Canada.}

\date{\today}

\maketitle
\begin{abstract} 
An association scheme is $P$-polynomial if and only if it consists of the distance matrices of a distance-regular graph. Recently, 
bivariate $P$-polynomial association schemes of type $(\alpha,\beta)$ were introduced by Bernard et al., and multivariate $P$-polynomial association schemes were later defined by Bannai et al. In this paper, the notion of $m$-distance-regular graph is defined and shown to give a graph interpretation of the multivariate $P$-polynomial association schemes. Various examples are provided. Refined structures and additional constraints for multivariate $P$-polynomial association schemes and $m$-distance-regular graphs are also considered. In particular, bivariate $P$-polynomial schemes of type $(\alpha, \beta)$ are discussed, and their connection to 2-distance-regular graphs is established.
\end{abstract}
\section{Background}

A connected undirected graph $G = (X,\Gamma)$ is said to be \textit{distance-regular} if for fixed $i$ and $j$, the number $p_{ij}^k$ of vertices $z \in X$ at distance $i$ from a vertex $x$ and at distance $j$ from a vertex $y$ depends only on the distance $k$ between $x$ and $y$.  This pivotal concept was introduced by Biggs \cite{biggs1971intersection, biggs1993algebraic} and led to significant results in graph theory and algebraic combinatorics \cite{brouwer2012distance}. In particular, these graphs were shown to be in one-to-one correspondence with $P$-polynomial association schemes \cite{BannaiBannaiItoTanaka+2021}.

Symmetric association schemes are central objects in algebraic combinatorics, arising in the theory of error-correction codes \cite{delsarte1998association} and the study of combinatorial designs \cite{BannaiBannaiItoTanaka+2021}. They can be defined as a set of $(0,1)$-matrices $\mathcal{Z} = \{A_0, A_1, \dots, A_N\}$ satisfying certain properties:
\begin{enumerate}
\item[(i)] $A_0 = I$, where $I$ is the identity matrix;
\item[(ii)] $\sum_{i = 0}^N A_i = J$, where $J$ is the matrix of ones;
\item[(iii)] $A_i = A_i^t$ for all $i \in \{0,1,\dots, N\}$;
\item[(iv)] The vector space spanned by these matrices is closed under matrix multiplication, i.e.\ there exist coefficients $p_{ij}^k$ such that the Bose-Mesner relations are verified,
\begin{equation}
A_i A_j = \sum_{k=0}^{N}p_{ij}^k A_k.
\end{equation}
\end{enumerate}
The matrices $A_i$ in a symmetric association scheme  commute and thus share common eigenspaces, with eigenvalues $\theta_{ij}$. The projectors $\{E_0, E_1,\dots, E_N\}$ onto these eigenspaces satisfy 
\begin{equation}
A_i E_j = \theta_{ij} E_j, \quad E_i E_j = \delta_{ij} E_i, \quad \sum_{i=0}^{N} E_i = I, \quad E_i^*=E_i,
\end{equation}
and offer an alternative basis for the Bose-Mesner algebra of $\mathcal{Z}$, i.e.\ $\langle A_0, A_1, \dots, A_N\rangle_\mathbb{C} = \langle E_0, E_1, \dots, E_N\rangle_\mathbb{C}.$

A scheme is said to be $P$-polynomial (with respect to the ordering of its matrices $A_i$ with the label $i$) if there exist polynomials $v_i(x)$ of degree $i$ such that
\begin{equation}
A_i = v_i(A_1), \quad \forall i \in \{0,1,\dots, N\}.
\end{equation}
It is known (see Proposition 1.1 of Section 3.1 in \cite{BannaiBannaiItoTanaka+2021}) that a set of matrices $\mathcal{Z}$ is a $P$-polynomial association scheme if and only if it is composed of the distance matrices $A_i$'s of a distance-regular graph $G$. In other words, the matrices $A_i$ are given by
\begin{equation}
(A_i)_{xy} = \left\{
	\begin{array}{ll}
		1  & \mbox{if } \text{dist}(x,y) = i,\\
		0 & \mbox{otherwise, }
	\end{array}
\right.
\end{equation}
where $\text{dist}(x,y)$ is the distance between $x$ and $y$, the length of a shortest path between $x$ and $y$ in $G$. $A_i$ is called the $i$-th distance matrix of $G$. 

The dual adjacency matrices $A_i^*$ of an association scheme are the diagonal matrices whose entries are defined in terms of an arbitrary fixed vertex $x_0$ and the projectors $E_i$, 
\begin{equation}
 (A_i^*)_{xx} = |X|(E_i)_{xx_0},
\end{equation}
where $X$ is the vertex set of the scheme. 
A scheme is said to be $Q$-polynomial if its dual adjacency matrices $A_i^*$ can be expressed as a polynomial $v_i^*$ of degree $i$ in $A_1^*$,
 \begin{equation}
A_i^* = v_i^*(A_1^*), \quad \forall i \in \{0,1,\dots, N\}.
\end{equation} 

While a complete classification of distance-regular graphs is yet to be achieved, a seminal result by Leonard \cite{leonard1982orthogonal,leonard1984parameters,BannaiBannaiItoTanaka+2021}
 showed that those whose associated  symmetric association schemes also have the $Q$-polynomial property are associated to hypergeometric orthogonal polynomials of the Askey scheme. 

The success of the $P$- and $Q$-polynomial properties in bridging graph theory, algebraic combinatorics and the study of special functions leaves one to wonder if a generalization of theses properties could lead to wider results on regular graphs and symmetric association schemes. In a recent work, Bernard et al. introduced the notion of \textit{bivariate $P$- and $Q$-polynomial association schemes of type $(\alpha,\beta)$} and showed that such schemes were related to families of bivariate polynomials verifying a special type of recurrence relations \cite{bernard2022bivariate}. They further gave multiple examples of association schemes which are not $P$- or $Q$-polynomial but are bivariate $P$- and/or $Q$-polynomial of type $(\alpha,\beta)$ for some $0\leq \alpha\leq 1$ and $0\leq \beta<1$. Some were notably related to multivariate extensions of polynomials in the Askey scheme \cite{bernard2022bivariate,crampe2023bivariate}.

As a generalization, Bannai,  Kurihara,
Zhao and
 Zhu put forward a definition of \textit{multivariate $P$- and $Q$-polynomial association schemes} \cite{bannai2023multivariate}, which deals with polynomials of arbitrary rank. It is worth noting that the approach is distinct to the one in \cite{bernard2022bivariate}, even for the bivariate case. Indeed, bivariate $P$- and $Q$-polynomial association schemes of type $(\alpha,\beta)$ require the introduction of a partial order $\preceq_{(\alpha,\beta)}$,   
which puts additional constraints on the subset $\mathcal{D} \subset \mathbb{N}^2$ of labels for the matrices of the scheme and on the associated bivariate polynomials $v_{ij}(x,y)$ such that
\begin{equation}
    A_{ij} = v_{ij}(A_{10}, A_{01}), \quad (i,j) \in \mathcal{D}.
\end{equation} 
 
As a result, the polynomials $v_{ij}(x,y)$ satisfy the same recurrences as the matrices $A_{ij}$'s of the scheme do. In contrast, the definition in \cite{bannai2023multivariate} only uses a total order (drops the use of partial orders), and puts less restrictions on the set of labels $\mathcal{D}$ as well as the polynomials, but introduces an additional condition on the adjacency algebra. This definition seems to encompass more association schemes and seems more natural for one investigating the properties of adjacency matrices. However, the polynomials may not satisfy the recurrence relations satisfied by  the matrices of the scheme. 

Now that the $P$- and $Q$-polynomial properties have been extended to higher rank, a natural question is to find a generalization of the correspondence between distance-regular graphs and $P$-polynomial association schemes. With this goal in mind, here we introduce a definition of \textit{$m$-distance-regular} graphs and show that it is in one-to-one correspondence with the $m$-variate $P$-polynomial association schemes defined in \cite{bannai2023multivariate}. We also investigate how the notions of $(\alpha,\beta)$-compatibility introduced in \cite{bernard2022bivariate} to study bivariate $P$-polynomial association schemes can be incorporated more generally within this framework. The paper is organised in the following way. In Section \ref{s2}, we recall the definition of multivariate $P$-polynomial association schemes and an equivalent description of such schemes in terms of  the intersection numbers. In Section \ref{s3}, we define $m$-distance-regular graphs and establish their correspondence with multivariate $P$-polynomial association schemes. Multiple examples are presented in Section \ref{s4}. In Section \ref{s5}, we refine the structure of multivariate $P$-polynomial association schemes and of $m$-distance-regular graphs using a partial order, and we examine some additional property on the set of labels of the adjacency matrices. This allows us to compare the association schemes considered in \cite{bernard2022bivariate} with those in \cite{bannai2023multivariate}, and to identify a class of $2$-distance-regular graphs that are in correspondence with the schemes lying in the intersection of the two types of schemes. Some open questions are given in Section \ref{s6}.

\section{Multivariate $P$-polynomial association schemes}\label{s2}

In this section, we recall the definition of multivariate $P$-polynomial association schemes introduced in \cite{bannai2023multivariate} and a proposition regarding their intersection numbers.

\begin{definition}\label{defmo} (Definition 2.2 of \cite{bannai2023multivariate}) A monomial order $\leq$ on $\mathbb{C}[x_1, x_2, \dots, x_m]$ is a relation on the set of
monomials $x_1^{n_1}x_2^{n_2}\dots x_m^{n_m}$ satisfying:
\begin{itemize}
\item[(i)] $\leq$ is a total order;
\item[(ii)] for monomials $u$, $v$, $w$, if $u \leq v$, then $wu \leq wv$;
\item[(iii)] $\leq$ is a well-ordering, i.e. any non-empty subset of the set of monomials has a minimum element under $\leq$.
\end{itemize}
\end{definition}
Since each monomial $x_1^{n_1}x_2^{n_2}\dots x_m^{n_m}$ can be associated to the tuple  $(n_1,n_2,\dots,n_m) \in \mathbb{N}^m$, we shall use the same order and notation $\leq$ on $\mathbb{N}^m = \{(n_1,n_2,\dots, n_m) \ | \ n_i \text{ are non-negative integers}\}$.

\begin{definition}\label{cpol} \cite{bannai2023multivariate} For a monomial order $\leq$ on $\mathbb{N}^m$, an $m$-variate polynomial $v(x_1, x_2, \dots, x_m) =
v(\x)$ is said to be of \textit{multidegree} $n \in \mathbb{N}^m$ if $v(\x)$ is of the form 
\begin{equation}
v(\x) = \sum_{a \leq n} f_a \x^a, \quad \x^a = x_1^{a_1} x_2^{a_2} \dots x_m^{a_m},
\end{equation}
with $f_n \neq 0$.
\end{definition}

Now, let $e_i \in \mathbb{N}^m$ be the tuple whose $i$-th entry is 1 and all the other entries are 0. 

\begin{definition}\label{defmv} (Definition 2.7. of \cite{bannai2023multivariate}) Let $\mathcal{D} \subset \mathbb{N}^m$ contain $e_1, e_2, \dots, e_m$ and $\leq$ be a monomial order on $\mathbb{N}^m$. A
commutative association scheme $\mathcal{Z}$ is called $m$-variate $P$-polynomial on the domain $\mathcal{D}$ with respect to $\leq$ if the following three conditions are satisfied:
\begin{enumerate}
\item[(i)] If $n = (n_1, n_2, \dots, n_m) \in \mathcal{D}$  and $0 \leq n_i' \leq  n_i$ for all $ i$, then $n' = (n_1', n_2', \dots, n_m') \in \mathcal{D}$;
\item[(ii)]  There exists a relabeling of the adjacency matrices of $\mathcal{Z}$:
\begin{equation}
\mathcal{Z} = \{A_n \ | \ n \in \mathcal{D}\}
\end{equation}
such that for all $n \in \mathcal{D}$ we have
\begin{equation}
A_n = v_n(A_{e_1}, A_{e_2}, \dots, A_{e_m})
\end{equation}
where $v_n(\boldsymbol{x})$ is an $m$-variate polynomial of multidegree $n$ and all monomials $\x^a$ in $v_n(\x)$ satisfy $a \in \mathcal{D}$;

\item[(iii)] For $i = 1,2, \dots, m$ and $a = (a_1, a_2, \dots, a_m) \in \mathcal{D}$, the product $A_{e_i}  A_{e_1}^{a_1}A_{e_2}^{a_2} \dots A_{e_m}^{a_m}$ is a
linear combination of
\begin{equation}
\{ A_{e_1}^{b_1}A_{e_2}^{b_2} \dots A_{e_m}^{b_m}\ | \ b = (b_1, \dots, b_m) \in \mathcal{D}, \ b \leq a + e_i  \}.
\end{equation}
\end{enumerate}
\end{definition}

In the following section, by use of edge partitions of a graph and a monomial order $\leq$ on $\mathbb{N}^m$, we define a new type of graphs with special combinatorial properties and show that they give rise to $m$-variate $P$-polynomial association schemes, as in Definition \ref{defmv}. To do so, we will use the following equivalent characterization of multivariate $P$-polynomial association schemes in terms of the intersection numbers. For completeness, a proof is presented in Appendix \ref{sec:appendix}.

\begin{proposition}\label{mprop}(Proposition 2.15 of \cite{bannai2023multivariate}) Let $\mathcal{D}\subset \mathbb{N}^m$ such that $e_1, e_2, \dots, e_m \in \mathcal{D}$ and $\mathcal{Z} = \{A_n \ | \ n \in \mathcal{D}\}$ be a commutative association scheme. Then the following two statements are
equivalent:
\begin{enumerate}
\item[(i)] $\mathcal{Z}$ is an $m$-variate $P$-polynomial association scheme on $\mathcal{D}$ with respect to a monomial order $\leq$;
\item[(ii)] The condition (i) of Definition \ref{defmv} holds for $\mathcal{D}$ and the intersection numbers satisfy,  
for each $i = 1,2,\dots, m$ and each $a \in \mathcal{D}$, $p_{e_i, a}^b \neq 0$ for $b \in \mathcal{D}$ implies $b \leq a + e_i$. Moreover, if $a + e_i \in \mathcal{D}$, then $p_{e_i, a}^{a + e_i} \neq 0$ holds.
\end{enumerate}
\end{proposition}

We now introduce the notion of $m$-distance and $m$-distance-regular graphs.

\section{$m$-distance-regular graphs}\label{s3}
Let $G = (X,\Gamma)$ be a connected undirected graph with vertex set $X$ and edge set $\Gamma \subseteq X\times X$. We consider an $m$-partition $\{\Gamma_i\ |\ i = 1,2,\dots, m\}$ of $\Gamma$, i.e.
\begin{equation}
\Gamma = \Gamma_1 \sqcup \Gamma_2 \sqcup \dots \sqcup \Gamma_m, \quad \Gamma_i \neq \emptyset, 
\end{equation}
where $\sqcup$ denotes disjoint union of sets. 
This can be interpreted as a colouring of the edges of $G$ with $m$ colors, with edges in $\Gamma_i$ colored $i$. It should not be confused with the definition of association schemes as colourings of the edges of the complete graph, see for instance \cite{van1999three}.

In the following, a \textit{walk} on $G$ is a finite ordered sequence of edges $\xi = (\gamma_1, \gamma_2, \dots, \gamma_L)$ such that $\gamma_i$ and $\gamma_{i+1}$ share a common vertex. Denote $\gamma_i$ as $\gamma_i= (y_i, y_{i+1}) \in \Gamma$. The vertices $y_1$ and $y_{L+1}$ constitute the two \textit{endpoints} of $\xi$. The endpoints of $\xi = \emptyset$ can be taken to be any $y \in X$. If for $i = 1,2,\dots, L$ the edge $\gamma_i$ has color $c_i$, that is, $\gamma_i \in \Gamma_{c_i}$, then we say that the walk $\xi = (\gamma_1,\gamma_2,\dots,\gamma_L)$ is of \textit{type} $c = (c_1,c_2,\dots, c_L)$.
\begin{definition}
Let $\xi = (\gamma_1,\gamma_2,\dots,\gamma_L)$ be a walk on $G$. Its\textit{ $m$-length} $\ell_m(\xi)$ with respect to the $m$-partition $\{\Gamma_i |\ i = 1,2,\dots, m\}$ is the vector in $\mathbb{N}^m$ whose $i$-th coordinate is the number of edges in the walk of  color $i$, that is, 
\begin{equation}
\ell_m(\xi) = (|\{j\ | \ \gamma_j \in \Gamma_1\}|, |\{j\ | \ \gamma_j \in \Gamma_2\}|, \dots, |\{j\ | \ \gamma_j \in \Gamma_m\}|). 
\end{equation}
\end{definition}
\begin{definition}
 Let $\leq$ be a monomial order on $\mathbb{C}[x_1, \dots, x_m]$.  The \textit{$m$-distance} $d_m$ between two vertices $x,y$ of a connected graph $G=(X,\Gamma = \Gamma_1 \sqcup  \dots \sqcup \Gamma_m)$ is defined as
 \begin{equation}
 d_m(x,y) = \text{min}_{\leq}\{\ell_m(\xi) \ | \ \xi \text{ is a walk between }x,y\}.
 \end{equation}
\end{definition}
This distance is well-defined because Definition \ref{defmo} requires that any non-empty subset of $\mathbb{N}^m$ has a minimum under $\leq$ and the set of walks between any two vertices is non-empty since $G$ is connected. We further note that the $m$-distance is symmetric (i.e.\ $d_m(x,y) = d_m(y,x)$) because $G$ is undirected so that any walk from $x$ to $y$ gives rise to a walk  from $y$ to $x$ when the order of edges is reversed. In the following, we denote by $\mathcal{D}$ the set of all $m$-distances between pairs of vertices in $G$,
\begin{equation}
\mathcal{D} = \{ \ell = (\ell_1,\dots, \ell_m) \in \mathbb{N}^m \ | \ \exists \ x,y \in X \ \text{s.t.}\ d_m(x,y) = \ell \}. \label{eq:distancesD}
\end{equation}
Note that $o:=(0,0,\dots, 0) \in \mathcal{D}$ since $d_m(x,x) =(0,0,\dots, 0)$ for all $x \in X$.  It is worth noting that $\Gamma_i \neq \emptyset$ does not imply that $e_i \in \mathcal{D}$. Indeed, some orders $\leq$ (e.g.\ the \textit{lex} order) allow the existence of non-zero vectors $a \notin \{e_j |\ j = 1,2,\dots,m\}$ satisfying $a \leq e_i$ for some $i \in \{1,2,\dots,m\}$. In this case two vertices connected by an edge with color $i$ may not be at $m$-distance $e_i$. 

We define the distance matrices associated with the $m$-distance $d_m$ of a graph as follows.
\begin{definition}
    Let $G = (X,\Gamma)$ be a graph with $m$-partition $\{\Gamma_i\ |\ i = 1,2,\dots, m\}$ and let $\mathcal{D}$ be the set of $m$-distances of $G$ with respect to a monomial order $\leq$ on $\mathbb{N}^m$ as in \eqref{eq:distancesD}. For $\ell \in \mathcal{D}$, the \textit{$\ell$-th $m$-distance matrix} $A_\ell$ of the graph $G$ is the matrix whose columns and rows are labelled by elements of $X$ and whose entries satisfy
\begin{equation}
(A_{\ell})_{xy} = \left\{
	\begin{array}{ll}
		1, & \mbox{if } d_m(x,y) = \ell,  \\
		0, & \mbox{otherwise.}
	\end{array}
\right.
\end{equation}
\end{definition}

Note that the $m$-distance matrices $A_\ell$ are symmetric, since the $m$-distance is a symmetric function of the vertices. In the case $m= 1$, the function $d_1$ reduces to the length of a shortest path in $G=(X,\Gamma)$ and the matrices $A_\ell$ correspond to the usual $\ell$-th distance matrix of $G$. 

We now define a generalization of the distance-regular property for a general $m \in \mathbb{N}\backslash\{0\}$ using the $m$-distance $d_m$. 

\begin{definition}\label{mdrg} A connected undirected graph is said to be \textit{$m$-distance-regular} if there exists an $m$-partition $\{\Gamma_i\ | \ i = 1,2,\dots,m\}$ of its edges and a monomial order $\leq$ such that $e_1$, $e_2, \dots, e_m$ are in the set $\mathcal{D}$ of all $m$-distances in $G$ and the number $p_{ab}^c$ of vertices $z$ at $m$-distance $a = (a_1, \dots, a_m)$ from a vertex $x$ and at $m$-distance $b= (b_1, \dots, b_m)$ from a vertex $y$ given that $x$ and $y$ are at $m$-distance $c = (c_1, \dots, c_m)$ does not depend on $x$ and $y$. 
\end{definition}

\begin{lemma} The set of $m$-distance matrices of an $m$-distance-regular graph forms a symmetric association scheme.
\end{lemma}
\begin{proof}
Let $G = (X,\Gamma)$ be an $m$-distance-regular graph with respect to the partition $\Gamma = \Gamma_1 \sqcup \Gamma_2 \sqcup \dots \sqcup \Gamma_m$ and the monomial order $\leq$. Let $\mathcal{D}$ be the set of all $m$-distances between vertices of $G$, $A_\ell$ be the $\ell$-th $m$-distance of $G$ for $\ell\in\mathcal{D}$, and let $\mathcal{Z} = \{A_\ell\ | \ \ell \in \mathcal{D}\}$. 
We show that $\mathcal{Z}$ is a symmetric association scheme. 
First, we observe from the definition of $m$-distance that $d_m(x,y) =o$ if and only if $x =y$ and hence $A_o = I$.  
Let $x$ and $y$ be two vertices of $G$ at $m$-distance $d_m(x,y) = c$. Then for any $a = (a_1,\dots,a_m)$ and $b = (b_1,\dots,b_m)$ in $\mathcal{D}$ we have
\begin{equation}
\begin{split}
(A_a A_b)_{xy} &= \sum_{z \in X}(A_a)_{xz}(A_b)_{zy}\\
&= \sum_{\substack{z \in X\\
d_m(x,z) = a\\
d_m(z,y) = b }}1 \cdot 1\\
& = p_{ab}^c = \Big(\sum_{\ell \in \mathcal{D}}  p_{ab}^\ell A_\ell\Big)_{xy}.
\end{split}
\end{equation}
This holds for all $x$ and $y$, regardless of their $m$-distance. Therefore, the Bose-Mesner relations are verified at the matrix level,
\begin{equation}
A_a A_b = \sum_{\ell \in \mathcal{D}}  p_{ab}^\ell A_\ell.
\end{equation}
The $m$-distance matrices are symmetric $A_\ell= A_\ell^t$ since the $m$-distance $d_m$ is. Because $G$ is connected and $\leq$ is a well-ordering, every pair of vertices has one and only one finite $m$-distance in $\mathcal{D}$, thus we further have
\begin{equation}
 \sum_{\ell \in \mathcal{D}} A_\ell=J,
\end{equation}
where $J$ is the matrix of ones. 
\end{proof}

\begin{corollary}\label{newco1} Let $x$ and $y$ be vertices in an $m$-distance-regular graph $G = (X,\Gamma)$. The number of walks of type $c = (c_1,c_2,\dots, c_L)$ between $x$ and $y$ is the same for all types $c'$ obtained as a permutation of $c$.
\end{corollary}
\begin{proof}
This follows from the commutativity of the matrices in a symmetric association scheme. Let $N_c$ be the number of walks of type $c=(c_1,\ldots, c_L) \in \{1,2,\dots,m\}^L$ between $x$ and $y$ and let $\pi$ be a permutation in the symmetric group $S_L$. It is straightforward using $A_{e_i}A_{e_j} = A_{e_j}A_{e_i}$ to check that
\begin{equation}
\begin{split}
N_c = \big( A_{e_{c_1}}A_{e_{c_2}}\dots A_{e_{c_L}}\big)_{xy}
=\big(A_{e_{\pi(c_1)}}A_{e_{\pi(c_2)}}\dots A_{e_{\pi(c_L)}}\big)_{xy} = N_{\pi(c)}.
\end{split}
\end{equation}
\end{proof}
\begin{remark}
    Note that in the proof of Corollary \ref{newco1}, it is important that the matrices $A_{e_i}$ belong to the set of $m$-distance matrices of the graph $G$, which justifies the requirement that the elements $e_1,\dots,e_m$ belong to the set of $m$-distances $\mathcal{D}$ in Definition \ref{mdrg}.
\end{remark}
Let $G$ be an $m$-distance regular graph, and assume $d_m(x,y)=a$. 
Note it is possible that for some $b\leq a$, there are no vertex pairs in $G$ which are at $m$-distance $b$. But there is more we can say if $b$ satisfies further conditions. 
\begin{lemma}\label{newco2}
Let $G = (X, \Gamma)$ be an $m$-distance regular graph and $x$ and $y$ be vertices at $m$-distance $a=(a_1,\ldots,a_m)$ in $G$. Then for any $b=(b_1,\ldots, b_m)\in \mathbb{N}^m$ such that $b_i\leq a_i$ for all $i=1,\ldots, m$, there exists a vertex $z \in X$ such that $d_m(x,z)=b$ and $d_m(y,z)=a-b$.
\end{lemma}
\begin{proof}
Assume $d_m(x,y)=a$ and $b \in \NN^m$ such that $b_i \leq a_i$. 
By Corollary \ref{newco1}, there exists a walk $\xi$ between $x$ and $y$ of $m$-length $a$ which decomposes as
\begin{equation}
\xi = \xi_1 \cdot \xi_2,
\end{equation}
where $\xi_1$ is a walk between $x$ and an intermediate vertex $z$ of $m$-length $\ell_m(\xi_1) = b$ and $\xi_2$ is a walk between $z$ and $y$ of $m$-length $\ell_m(\xi_2) = a - b$.

We claim that $d_m(x,z)=b$ and $d_m(z,y)= a -b$. First, we know that $d_m(x,z)\leq b$ and $d_m(z,y)\leq a - b$ since these $m$-distances are bounded above by the $m$-lengths of $\xi_1$ and $\xi_2$, respectively. Next, assume that $d_m(x,z) <  b$. Then there exists a walk $\xi_3$ between $x$ and $z$ of $m$-length $\ell_m(\xi_3) < b$. But then the walk $\xi_3 \cdot \xi_2$ between $x$ and $y$ is of $m$-length  $\ell_m(\xi_3 \cdot \xi_2) < a$,  which contradicts the initial assumption that $d_m(x,y) = a$. A similar contradiction is obtained if we assume $d_m(z,y) < a -b$, which completes the proof.
\end{proof}

\begin{corollary} Let $G$ be an $m$-distance regular graph and let $x,y$ be two vertices at $m$-distance $a$. If $z$ is a vertex on a walk $\xi$ of $m$-length $\ell_m(\xi)=a$ between $x$ and $y$, then $d_m(x,z)=b$ and $d_m(y,z)=a-b$ for some $b\in \mathbb{N}^m$ with $b_i\leq a_i$ for all $i = 1,2,\dots, m$. 
\end{corollary}
\begin{proof}
It follows from the proof of Lemma \ref{newco2}.
\end{proof}

Like distance-regular graphs,  certain intersection numbers of an $m$-distance-regular graph must be zero. 
\begin{lemma}\label{lem1} The intersection numbers of the symmetric association scheme associated to an $m$-distance-regular graph satisfy
\begin{equation}
p_{ab}^c \neq 0 \quad \Rightarrow \quad a \leq b + c, \ b \leq a + c, \ c \leq a + b.
\end{equation}
\end{lemma}
\begin{proof}
$p_{ab}^c\neq 0$ if and only if there exist three vertices $x,y,z\in X$ such that $d_m(x,y)=c$,  $d_m(x,z)=a$ and $d_m(z,y)=b$. Therefore there is a walk $\xi$ between $x$ and $y$ of $m$-length $\ell_m(\xi) = a+b$. By the definition of $m$-distance, we know that $d_m(x,y) \leq \ell_m(\xi) $ and thus $c  \leq a + b$. A similar argument about $y$ and $z$ or $x$ and $z$ shows the two other inequalities. 
\end{proof}

The previous lemma imposes important conditions on the possible non-zero intersection numbers of the association scheme associated to an $m$-distance-regular graph. We can also show that certain intersection numbers cannot be zero.

\begin{lemma}\label{lem2} Let $p_{ab}^c$ be the intersection numbers of the  symmetric association scheme associated to an $m$-distance-regular graph. If $a,b,c \in \mathcal{D}$, the set of $m$-distances of $G$, then $p_{ab}^c \neq 0$ if one of $a,b,c$ is equal to the sum of the other two.
\end{lemma}
\begin{proof}
If $a+b=c$, then there exist two vertices $x$ and $y$ such that $d_m(x,y)=c$. By Lemma \ref{newco2}, there exists a vertex $z$ such that $d_m(x,z)=a$ and $d_m(y,z)=b$. The distances between the triple $\{x,y,z\}$ imply that $p_{ab}^{c}\neq 0$, $p_{ac}^{b}\neq 0$, and $p_{bc}^{a}\neq 0$.
\end{proof}

We have now gathered the necessary findings to prove our main result which provides a correspondence between the $m$-distance-regular property of graphs and the $m$-variate $P$-polynomial property of schemes. The theorem states that a connected undirected graph $G = (X,\Gamma)$ is $m$-distance-regular if and only if there exists an $m$-variate $P$-polynomial association scheme $\mathcal{Z} = \{A_n | \ n \in \mathcal{D}\}$ such that the adjacency matrix of $G$ is $A = A_{e_1} + A_{e_2} + \dots + A_{e_m}$. 

\begin{theorem}\label{realdeal}
Let $\mathcal{Z} = \{ A_n|\ n \in \mathcal{D}\}$ be a symmetric association scheme on a set $X$ with domain $\mathcal{D} \subset \mathbb{N}^m$ containing $e_1,e_2, \dots, e_m$. Let $\leq$ be a monomial order on $\mathbb{N}^m$. For $i = 1,2,\dots,m$, let $\Gamma_i\subseteq X\times X$ be defined by
\begin{equation} \label{eq:edgeai}
(x,y) \in \Gamma_i \quad \iff \quad (A_{e_i})_{xy} = 1.
\end{equation}
The following statements are equivalent:
\begin{enumerate}
\item[(i)]  The graph $G = (X,\Gamma = \Gamma_1 \sqcup \dots \sqcup  \Gamma_m )$ is $m$-distance-regular with respect to the partition $\{\Gamma_i\ | \ i = 1, 2, \dots, m\}$ of its edges and the monomial order $\leq$. The subset $\mathcal{D}$ is the set of all $m$-distances in $G$ and the matrix $A_\ell$ is the $\ell$-th $m$-distance matrix of $G$ for all $\ell \in \mathcal{D}$;
\item[(ii)] $\mathcal{Z}$ is a symmetric $m$-variate $P$-polynomial association scheme on $\mathcal{D}$ with respect to the monomial order $\leq$.
\end{enumerate}
\end{theorem}
\begin{proof}
$(i) \Rightarrow (ii)$. We show that an $m$-distance-regular graph $G$ leads to a symmetric $m$-variate $P$-polynomial scheme. Using Lemma \ref{newco2}, we find that condition (i) of Definition \ref{defmv} holds for the set $\mathcal{D}$ of $m$-distances in $G$. 
Lemmas \ref{lem1} and \ref{lem2} imply that the intersection numbers $p_{ab}^c$ of the scheme associated to the $m$-distance-regular graph $G$ satisfy the conditions in (ii) of Proposition \ref{mprop}. Therefore the $m$-distance matrices of an $m$-distance-regular graph form an $m$-variate $P$-polynomial association scheme.\\

$(ii) \Rightarrow (i)$. Assume $\mathcal{Z} = \{ A_n|\ n \in \mathcal{D}\}$ is a symmetric $m$-variate $P$-polynomial association scheme. 
Let $G$ be the graph whose adjacency matrix is 
\begin{equation}
A = A_{e_1} + A_{e_2} +\dots +A_{e_m}. 
\end{equation}
We prove that $G=(X, \Gamma)$ is an $m$-distance regular graph with respect to the partition $\Gamma = \Gamma_1 \sqcup \Gamma_2 \sqcup \dots \sqcup \Gamma_m$, where $\Gamma_i$'s are as in \eqref{eq:edgeai}. 

$G$ is undirected since the matrices $A_{e_i}$ are symmetric. 
We now show that for each $\ell\in \mathcal{D}$, the matrix $A_\ell \in \mathcal{Z}$ is in fact the $\ell$-th $m$-distance matrix of $G$.  Let $x$ and $y$ be two vertices in $X$ such that $(A_\ell)_{xy} = 1$. Using the following property which holds for matrices in $m$-variate $P$-polynomial schemes and $b \in \mathcal{D}$ (as proved in \cite{bannai2023multivariate}, see also Lemma \ref{lem:span} in Appendix \ref{sec:appendix}),
\begin{equation}\label{vseq}
\text{span}\{A_a \ | \ a \leq b, \ a \in \mathcal{D} \} = \text{span}\Big\{\prod_{i = 1}^m A_{e_i}^{a_i} \ | \ a \leq b,\ a \in \mathcal{D} \Big\},
\end{equation}
we can show that $d_m(x,y) \leq \ell $. Indeed, equation \eqref{vseq} implies the existence of coefficients $f_a$ such that
\begin{equation}
1 = (A_\ell)_{xy}= \sum_{\substack{a \in \mathcal{D}\\ a \leq \ell}} f_a \Big(\prod_{i = 1}^m A_{e_i}^{a_i}\Big)_{xy}.
\end{equation}
Thus there exists $a \in \mathcal{D}$ such that $a \leq \ell$ and $\left(\prod_{i = 1}^m A_{e_i}^{a_i}\right)_{xy} \neq 0$. That is, there is a walk of $m$-length $a \leq \ell$ between $x$ and $y$. By the definition of $m$-distance, $d_m(x,y) \leq \ell$. 

Now we show $d_m(x,y) =\ell$ by contradiction. Assume $d_m(x,y) = a < \ell$. Then there exists a walk $\xi$ of $m$-length $\ell_m(\xi) = a$ between $x$ and $y$ and hence $\big(\prod_{i = 1}^m A_{e_i}^{a_i}\big)_{xy}\neq 0$. There are two cases to consider. If $a \in \mathcal{D}$, we can use equation \eqref{vseq} again and find coefficients $f'_b$ such that
\begin{equation}
0 \neq \big(\prod_{i = 1}^m A_{e_i}^{a_i}\big)_{xy} = \big( \sum_{\substack{b \in \mathcal{D} \\ b \leq a}}f'_b A_b \big)_{xy}.
\end{equation}
This implies that $(A_b)_{xy} =1 $ for some $b \in \mathcal{D}$ with $b \leq a < \ell$. It is a contradiction since the matrix $A_\ell$ is the unique matrix in the scheme whose $xy$ entry is non-zero.  If $a \notin \mathcal{D}$, there exist $c \in \mathcal{D}$ and $c' \in \mathbb{N}^m \backslash \{o\}$ whose sum is $a = c + c'$. For $e_k$  such that $c' - e_k \in \mathbb{N}^m$, the condition (iii) of Definition \ref{defmv} gives the existence of coefficients $g_b$ such that
\begin{equation}\label{ex1}
\begin{split}
 \prod_{i = 1}^m A_{e_i}^{a_i}  &=  \prod_{i = 1}^m A_{e_i}^{c'_i} \prod_{j = 1}^m A_{e_j}^{c_j} \\
 &=  \Big(\prod_{i = 1}^m A_{e_i}^{(c' - e_k)_i}  \Big)\Big(A_{e_k}\prod_{j = 1}^m A_{e_j}^{c_j }\Big)\\
 &= \Big(\prod_{i = 1}^m A_{e_i}^{(c' - e_k)_i}  \Big)\Big(\sum_{\substack{b \in \mathcal{D} \\ b \leq c + e_k}}g_b \prod_{j = 1}^m A_{e_j}^{b_j} \Big)\\
 & = \sum_{\substack{b \in \mathcal{D} \\ b \leq c + e_k}}g_b \prod_{i = 1}^m A_{e_i}^{(c' - e_k)_i } \prod_{j = 1}^m A_{e_j}^{b_j}.
 \end{split}
\end{equation}
If $c' - e_k \neq o$, we can find again $e_{k'}$ such that $c' - e_k - e_{k'}\in \mathbb{N}^m$. One then observes that the argument based on condition (iii) and illustrated in \eqref{ex1} can be applied again for each term in the sum of the last line. Using this approach iteratively, one eventually finds coefficients $g'_b$ such that
\begin{equation}\label{eqre2}
\begin{split}
 \prod_{i = 1}^m A_{e_i}^{a_i}  &=  \sum_{\substack{b \in \mathcal{D}\\
 b \leq c + c'}} g'_b\prod_{j = 1}^m A_{e_j}^{b_j}.
 \end{split}
\end{equation}
From equation \eqref{eqre2} and $(\prod_{i = 1}^m A_{e_i}^{a_i})_{xy} \neq 0$, we find that there is $b \in \mathcal{D}$ with $b < c + c'  = a$ such that $(\prod_{j = 1}^m A_{e_j}^{b_j })_{xy} \neq 0$. The strict inequality between $b$ and $a$ comes from the fact that $b \in \mathcal{D}$ but $a \notin \mathcal{D}$. Hence, there exists a walk of $m$-length $b < a$ between $x$ and $y$, which contradicts $d_m(x,y) = a$. Therefore $d_m(x,y)=\ell$, and $A_\ell\in\mathcal{Z}$ is the $\ell$-th $m$-distance matrix of $G$. In particular, the elements $e_1,e_2, \dots, e_m$ belong to the set of $m$-distances of $G$.

Note that the connectivity of $G$ also follows from the previous argument. Indeed, for any pair of vertices $x$ and $y$ there is a finite $m$-distance $\ell \in \mathcal{D}$ such that $(A_\ell)_{xy} = 1$ and thus $d_m(x,y) = \ell$. Therefore, there is a finite walk in $G$ connecting $x$ and $y$.
  
To prove that $G$ is $m$-distance-regular, we can use the Bose-Mesner relations satisfied by the matrices in an association scheme,
\begin{equation}
A_a A_b = \sum_{c \in \mathcal{D}} p_{ab}^c A_c.
\end{equation}
For any $x$ and $y$ at $m$-distance $d_m(x,y) = c$ and any $a,b\in\mathcal{D}$, we find that
\begin{equation}
\begin{split}
&  |\{ z \in X \ | \ d_m(x,z) = a, \ d_m(z,y) = b\}|\\
=&  \sum_{z \in X} (A_a)_{xz} (A_b)_{zy}\\
=& (A_aA_b)_{xy} =\Big(\sum_{\ell\in\mathcal{D}} p_{ab}^\ell A_\ell\Big)_{xy} \\
=&p_{ab}^c .
\end{split}
\end{equation}
This shows that the number $p_{ab}^c$ of vertices $z$ at $m$-distance $a$ from $x$ and at $m$-distance $b$ from $y$ depends only on the $m$-distance $c$ between $x$ and $y$. Therefore $G$ is $m$-distance-regular.
\end{proof}

\begin{corollary}\label{cor:orgen} There exists at most one $m$-variate $P$-polynomial association scheme $\mathcal{Z}$ with order $\leq$ and with fixed  generating matrices  $A_{e_i}$, $i= 1,2\dots, m$.
\end{corollary}
\begin{proof}
From Theorem \ref{realdeal}, we know that $\mathcal{Z}$ is such an $m$-variate $P$-polynomial association scheme if and only if the graph $G$ with adjacency matrix $A = A_{e_1} + A_{e_2} + \dots + A_{e_m}$ is $m$-distance-regular and that 
$\mathcal{Z}$ is composed of the $m$-distance matrices $A_\ell$ of $G$. Therefore, $\mathcal{Z}$ is uniquely defined by the order $\leq$ and the generating  matrices $A_{e_i}$, $i = 1,2,\dots,m$. 
\end{proof}

\section{Examples}\label{s4}

This section provides a few examples of $m$-distance-regular graphs and illustrates their relation to $m$-variate $P$-polynomial association schemes.

\subsection{Cartesian product of distance-regular graphs} \label{s4subs1}

For each $i = 1,2,\dots, m$, let  $G_i$ be a distance-regular graph and denote its $\ell$-th distance matrix by $A_\ell^{(i)}$. The adjacency matrix $\mathcal{A}$ of the Cartesian product of these graphs $\mathcal{G} = G_1 \square G_2 \square \dots \square G_m$ is   given by
\begin{equation}\label{cp1}
\mathcal{A} = \sum_{i = 1}^m I^{(1)} \otimes I^{(2)}\otimes \dots \otimes A_1^{(i)}\otimes \dots \otimes I^{(m)},
\end{equation} 
where $I^{(i)} = A_0^{(i)}$ is the identity matrix whose rows and columns are indexed by the vertices of $G_i$. An $m$-partition $\{\Gamma_i \ | \ i = 1,2,\dots,m\}$ can be introduced using the terms in the sum \eqref{cp1},
\begin{equation}
 (x,y) \in \Gamma_i \quad \iff \quad (I^{(1)} \otimes I^{(2)}\otimes \dots \otimes A_1^{(i)}\otimes \dots \otimes I^{(m)})_{xy} \neq 0,
\end{equation}
where $x$ and $y$ are vertices in $\mathcal{G}$. This $m$-partition can be seen as colouring the edges of $\mathcal{G}$ according to the graph $G_i$ they originate from. It can be shown that this partition makes $\mathcal{G}$ an $m$-distance-regular graph with respect to any monomial order $\leq$. Indeed, let us consider two vertices $x = (x_1,x_2, \dots, x_m)$ and $y = (y_1,y_2,\dots,y_m)$ of $\mathcal{G}$. For any $\leq$, they are at $m$-distance
\begin{equation}
d_m(x,y) = (d_1(x_1,y_1), d_1(x_2,y_2),\dots, d_1(x_m,y_m)),
\end{equation}
where $d_1(x_i,y_i)$ is the length of a shortest path between the vertices $x_i$ and $y_i$ in the graph $G_i$. Let us denote $d_m(x,y) = c$. One observes that the number $p_{ab}^c$ of vertices $z = (z_1,z_2, \dots, z_m)$ at $m$-distance $a = (a_1, a_2, \dots, a_m)$ from $x$ and $b = (b_1, b_2, \dots, b_m)$ from $y$ is given by
\begin{equation}
p_{ab}^c = \prod_{i = 1}^{m} {p} _{a_i b_i}^{G_i, c_i},
\end{equation} 
where ${p}_{a_i b_i}^{G_i, c_i}$ are the intersection numbers of the distance-regular graph $G_i$. Therefore  $p_{ab}^c$ does not depend on $x$ and $y$. 
For each $j=1,2,\dots,m$, let $x^{(j)}=(x_1,\ldots, x_m)$ and $ y^{(j)}=(y_1,\ldots, y_m)$ be two vertices of $\mathcal{G}$ such that $x_i=y_i$ for $i\neq j$ and $d_1(x_j,y_j)=1$ in $G_j$. We have $d_m(x^{(j)},y^{(j)})=e_j$, therefore $e_j$ is an $m$-distance of $\mathcal{G}$ for all $j=1,\ldots, m$. 
The graph $\mathcal{G}$ is thus by definition $m$-distance-regular. 

The set of all $m$-distances of $\mathcal{G}$ is 
\begin{equation}
\mathcal{D} = \{\ell \in \mathbb{N}^m \ | \ \ell_i \leq \text{diameter of }G_i\}.
\end{equation} 
Therefore the $m$-variate $P$-polynomial association scheme associated to the $m$-distance-regular graph $\mathcal{G}$ is composed of the matrices
\begin{equation}
\mathcal{A}_\ell = A^{(1)}_{\ell_1} \otimes A^{(2)}_{\ell_2}\otimes \dots \otimes A^{(m)}_{\ell_m}, \quad \ell = (\ell_1,\ell_2,\dots,\ell_m) \in \mathcal{D}.
\end{equation}
This scheme notably corresponds to the direct product of the $P$-polynomial association schemes 
associated to the distance regular graphs $G_1$, \ldots, $G_m$, respectively,  
and hence is an $m$-variate $P$-polynomial association scheme with respect to any monomial order, as discussed in \cite{bernard2022bivariate,bannai2023multivariate}. 
This and Theorem~\ref{realdeal} provide another proof that the Cartesian product of $m$ distance-regular graphs is an $m$-distance-regular graph, with respect to any monomial order. 

Figure \ref{fig:CP1} illustrates how toroidal square lattices of dimension $m$ are $m$-distance-regular, since they are obtained by taking the Cartesian product of $m$ cycle graphs, which are distance-regular.

\begin{figure}[h]
\vspace*{-.25cm}
\begin{subfigure}{.5\textwidth}
  \centering
  \hspace*{1cm}
  \includegraphics[width=1\linewidth]{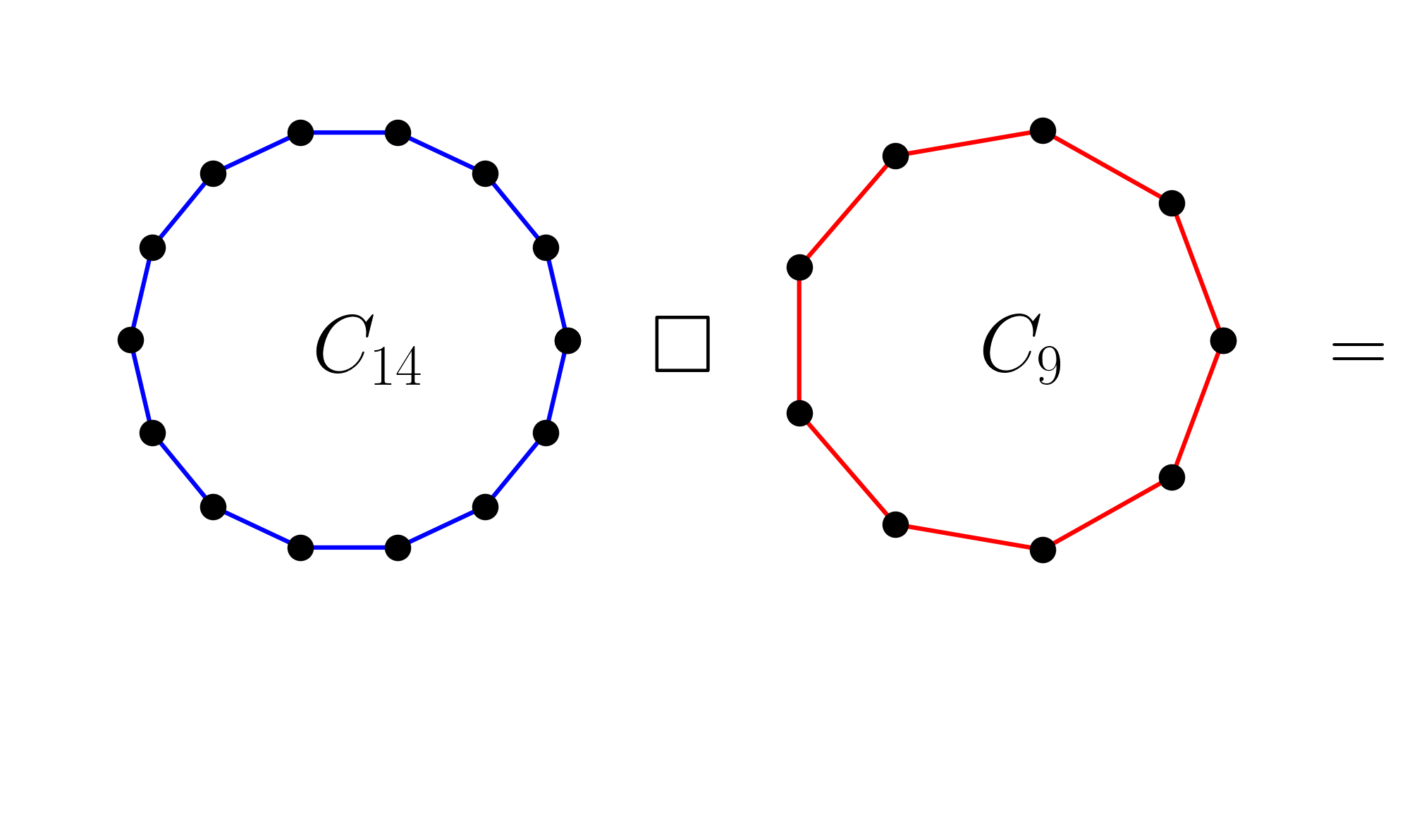}
  \label{fig:sfig1}
\end{subfigure}%
\begin{subfigure}{.5\textwidth}
  \centering
  \hspace*{-1cm}
  \includegraphics[width=.7\linewidth]{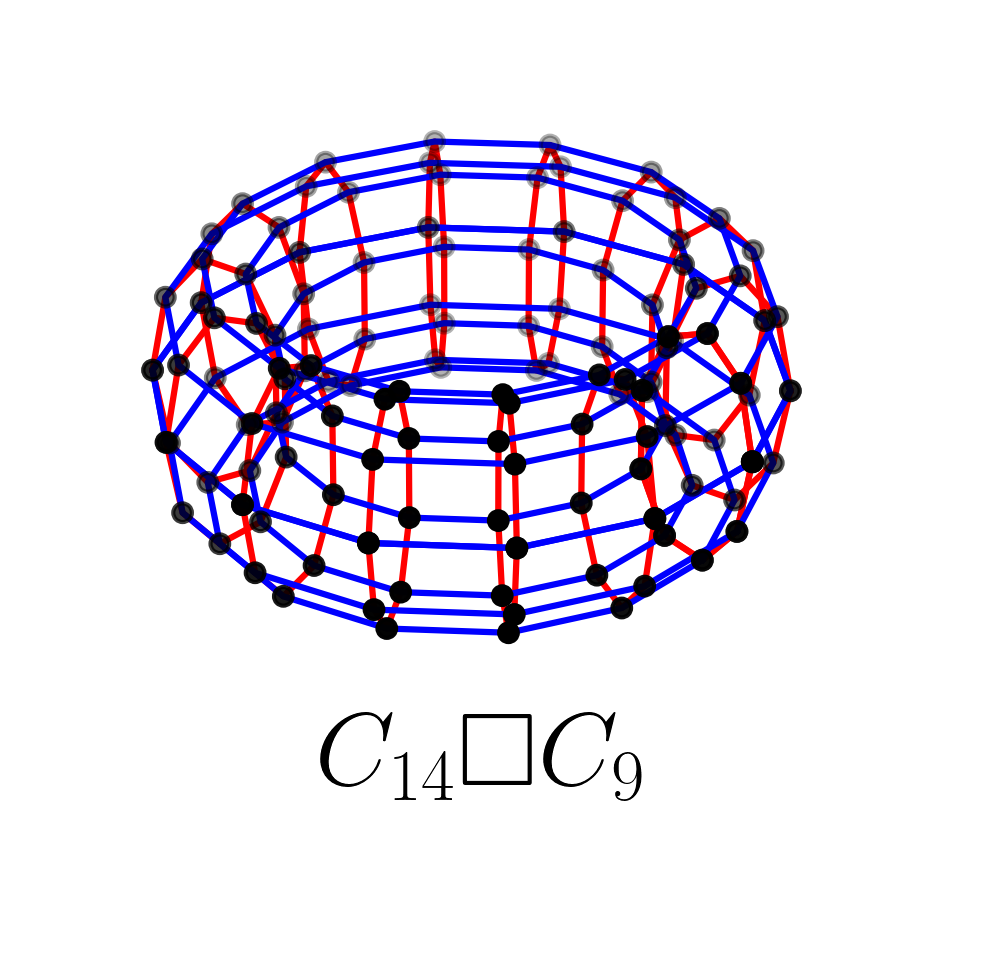}
  \label{fig:sfig2}
\end{subfigure}
\vspace*{-1.2cm}
\caption{Cartesian product of the cycle graphs $C_{14}$ and $C_{9}$. The result is a two-dimensional toroidal square lattice, which is $2$-distance-regular with respect to the blue and red colouring of its edges and any monomial order.}
\label{fig:CP1}
\end{figure}

\subsection{Hamming graphs and the symmetrization of association schemes}

Let $\mathcal{Z} = \{A_0, A_1, \dots, A_m\}$ be a commutative association scheme with vertex set $X$. For  a positive integer $k$, its \textit{symmetrization} or \textit{extension} of length $k$ refers to the set
\begin{equation}
\mathcal{S}^k(\mathcal{Z}) = \{ \mathcal{A}_{n} \ | \ n \in \mathcal{D}\},
\end{equation}
where
\begin{equation}
 \mathcal{D} = \{ n = (n_1,n_2,\dots,n_m) \in \mathbb{N}^m \ |\  |n| \leq k\}, \quad |n| = \sum_{i = 1}^m n_i
\end{equation}
and
\begin{equation}
\mathcal{A}_n = \frac{1}{(n_1)! (n_2)! \dots (n_m)! (k-|n|)!} \sum_{\pi \in S_k} \pi \cdot A_1^{\otimes^{n_1}}\otimes A_2^{\otimes^{n_2}} \otimes \dots \otimes A_m^{\otimes^{n_m}} \otimes A_0^{\otimes^{k - |n|}}, 
\end{equation}
where the sum is over all the place permutations, and the prefactor ensures that each term appears only once. 
In the case $m = 2$, the set $\mathcal{S}^k(\mathcal{Z})$ was shown to be a bivariate $P$- and $Q$-polynomial association scheme of type $(1/2,1/2)$ \cite{bernard2022bivariate}. For a general $m \in \mathbb{N}\backslash\{0\}$, it was further shown to be an $m$-variate $P$- and $Q$-polynomial association scheme \cite{bannai2023multivariate} with respect to the \textit{deg-lex} order $\leq$ defined  by
\begin{equation}\label{CPord}
 a \leq b \iff \left\{
	\begin{array}{ll}
		 \sum_{i=1}^m a_i < \sum_{i=1}^m b_i, \\
		 \text{or} \\
		\sum_{i=1}^m a_i = \sum_{i=1}^m b_i \ \text{and the leftmost non-zero entry of $a-b \in \mathbb{Z}^{m}$ is negative.}
	\end{array}
\right. 
\end{equation}
From this result and Theorem \ref{realdeal}, we find that the graph $G=(X,\Gamma)$ whose adjacency matrix is
\begin{equation}
A = \sum_{i = 1}^{m} \mathcal{A}_{e_i},
\end{equation}
is $m$-distance-regular with respect to deg-lex order $\leq$ and the partition of its edges $\Gamma = \Gamma_1 \sqcup \Gamma_2 \sqcup \dots \sqcup \Gamma_m$ with 
\begin{equation}
 (x,y) \in \Gamma_i \quad \iff \quad (\mathcal{A}_{e_i})_{xy} = 1.
\end{equation}
 Since $\mathcal{Z}$ is an association scheme, its matrices have the property of summing to the matrix of ones $J$. It follows that the adjacency matrix $A$ of $G$ can be rewritten as
\begin{equation}
A = \sum_{i = 1}^k \underbrace{I \otimes \dots \otimes I}_{i-1 \text{ times}}\otimes (J- I) \otimes \underbrace{I \otimes \dots \otimes I}_{k - i \text{ times}},
\end{equation}
where $I$ is the identity matrix of dimension $|X|$. This matrix can be identified as the adjacency matrix of the Hamming graph $H(k, |X|)$, leading to the following proposition:
\begin{proposition}\label{lemSYM} 
Let $\mathcal{Z}=\{A_0,A_1,\ldots, A_m\}$ be any $m$-class association scheme on $q$ vertices. Let $\mathcal{S}^k(\mathcal{Z})$ be its extension scheme of length $k$. Then the $m$-distance-regular graph associated to the $m$-variate $P$-polynomial association scheme $\mathcal{S}^k(\mathcal{Z})$ is the Hamming graph $H(k,q)$. In particular, 
the Hamming graph $H(k,q)$ is $m$-distance-regular if there exists an $m$-class association scheme on $q$ vertices.
\end{proposition}

Let $I$ be the identity matrix of size 2 and let $\sigma_x=J-I$. 
Then $\mathcal{Z} = \{I \otimes I, I \otimes \sigma_x + \sigma_x \otimes I, \sigma_x \otimes \sigma_x \} $ is a two-class association scheme on 4 vertices. 
Therefore the Hamming graph $H(k,4)$ is 2-distance regular for any integer $k\geq 2$. The edge partitions of $H(2,4)$ and $H(3,4)$ into two parts with respect to which and the deg-lex order the graphs are 2-distance-regular are shown in  Figure \ref{fig2}.\\

\noindent \textbf{Remark}: For any integer $q\geq 2$, the collection of the two $q\times q$ matrices $\{I, J-I\}$ is a one-class association scheme, and 
 Proposition \ref{lemSYM} reduces to the well-known result that Hamming graphs are distance-regular. The Hamming graph $H(k,q)$ can also be constructed as the $k$-th Cartesian power of the complete graph $K_q$ (which is  distance-regular). The discussion in Section~\ref{s4subs1} implies that $H(k,q)$ is $k$-distance-regular. 
Therefore the opposite implication of the last statement of the Proposition is not true, i.e.,  $H(k,q)$ is  $m$-distance regular does not imply the existence of an $m$-class association scheme on $q$ vertices. For example, $H(2,7)=K_7\square K_7$ is 2-distance-regular, but there are no strongly-regular graphs (association schemes with 2 classes) on 7 vertices \cite{Brouwer2022sgr}.   

The symmetrization of a Hamming scheme is referred to as an \textit{ordered} Hamming scheme. As pointed out in this subsection, $m$-distance-regular graphs associated to the symmetrization of a scheme are Hamming graphs. It follows that the $m$-distance-regular graphs associated to ordered Hamming schemes are Hamming graphs.

\begin{figure}[h]
\begin{subfigure}{.5\textwidth}
  \centering
  \includegraphics[width=1\linewidth]{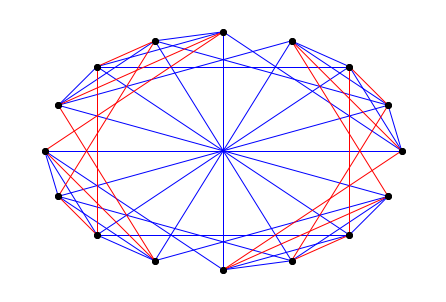}
  \label{fig:sfig1}
\end{subfigure}%
\begin{subfigure}{.5\textwidth}
  \centering
  \includegraphics[width=1\linewidth]{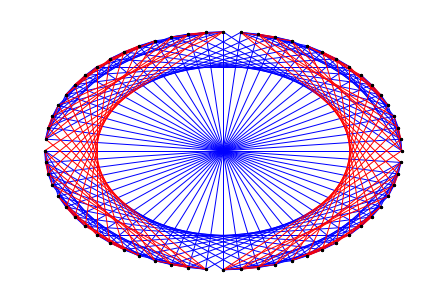}
\end{subfigure}
\vspace*{-1cm}
\caption{ Hamming graphs $H(2,4)$ on the left and $H(3,4)$ on the right. They are $2$-distance-regular with respect to the blue and red partition of the edges and the deg-lex order. This is the partition obtained using the symmetrizations $\mathcal{S}^2(\mathcal{Z})$  and $\mathcal{S}^3(\mathcal{Z})$ of $\mathcal{Z} = \{I \otimes I, I \otimes \sigma_x + \sigma_x \otimes I, \sigma_x \otimes \sigma_x \} $.}
  \label{fig2}
\end{figure}

\subsection{24-cell}\label{subse:24cell}

Consider the set $X$ of vectors in $\mathbb{R}^4$ obtained by permutations of $(\pm 1, \pm 1, 0 ,0)$. Its cardinality is $24$ and any two of its elements are at a Euclidean distance $0$, $\sqrt{2}$, $2$, $\sqrt{6}$ or $\sqrt{8}$ as vectors in $\mathbb{R}^4$.  Let $G = (X,\Gamma)$ be the graph with edge set $\Gamma  = \Gamma_1 \sqcup \Gamma_2$ where for $x,y\in X$, 
\begin{equation}
(x,y) \in \Gamma_1 \ \iff \ \|x-y\| = 2,
\end{equation}
\begin{equation}
(x,y) \in \Gamma_2 \ \iff \ \|x-y\| = \sqrt{6}.
\end{equation}
This graph is $2$-distance-regular with respect to the partition $\{\Gamma_1, \Gamma_2\}$ and the order $\leq$ defined in equation \eqref{CPord}. Indeed, one can check that this partition and order implies a correspondence between the $2$-distances in $G$ and the usual distance between elements of $X$ in $\mathbb{R}^4$:
\begin{equation}
 d_2(x,y) = (0,0) \ \iff \ \|x-y\| = 0,
\end{equation}

\begin{equation}
 d_2(x,y) = (1,0) \ \iff \ \|x-y\| = 2,
\end{equation}

\begin{equation}
 d_2(x,y) = (0,1) \ \iff \ \|x-y\| = \sqrt{6},
\end{equation}

\begin{equation}
 d_2(x,y) = (0,2) \ \iff \ \|x-y\| = \sqrt{2},
\end{equation}

\begin{equation}
 d_2(x,y) = (2,0) \ \iff \ \|x-y\|= \sqrt{8}.
\end{equation}
Furthermore, one can use rotations and reflections in $\mathbb{R}^4$ to show that the number of vectors $z$ in $X $ such that $\|x-z\| =r_1$ and $\|y-z\| =r_2$ given that $\|x-y\| = r_3$ depends on $r_i$ but not on $x$ and $y$. 

The graph $G=(X,\Gamma  = \Gamma_1 \sqcup \Gamma_2)$ is closely related to the convex regular 4-polytope known as \textit{$24$-cell}, which is the convex hull of $X$ and is associated to a univariate $Q$-polynomial association scheme \cite{Billnotes}. The skeleton of this polytope is the graph on $X$, where $x$ and $y$ are adjacent if and only if $\|x-y\| = \sqrt{2}$. In other words, the skeleton graph of the $24$-cell is the $2$-distance $d_2(x,y) = (0,2)$ graph of the $2$-distance-regular graph $G$. Alternatively, one could interpret $G$ as the distance-2 graph of the skeleton graph of the $24$-cell.

\section{Refinements and additional conditions}\label{s5}
In \cite{bernard2022bivariate}, the bivariate polynomial structures of symmetric association schemes constrained by both a monomial order $\leq$ and a partial order $\poo$ are studied. The benefit of this partial order is to restrict the type of recurrence relations satisfied by the polynomials arising from the association schemes. As shown in \cite{bernard2022bivariate}, several natural examples of association schemes involving bivariate polynomials do indeed possess a structure which is more refined than what is obtained using only a total order $\leq$. The extra properties induced by the additional partial order are therefore relevant, as they may lead to a more attainable classification of known bivariate schemes. 

Motivated by this, we now refine the definitions of multivariate $P$-polynomial association schemes and distance-regular graphs considered in the previous sections using a partial order, with which the original monomial order is compatible, and show that they are in correspondence. Conditions concerning the domain and its boundary are also discussed. These refinements and conditions allow us to examine the relations between the two definitions of bivariate $P$-polynomial association schemes in 
\cite{bernard2022bivariate} and \cite{bannai2023multivariate}, 
and to obtain a graph interpretation of the bivariate $P$-polynomial association schemes of type $(\alpha,\beta)$ as well. The differences between the definitions of \cite{bernard2022bivariate} and of \cite{bannai2023multivariate} are explained and illustrated with an example. 

\subsection{Partial order refinement}\label{ssec:porefin}
Let $\leq$ be a monomial order on $\NN^m$ and let $\po$ be a partial order on $\NN^m$ with the following additional properties:
\begin{itemize}
\item[(i)] for all $a,b \in \NN^m$, if $a \po b$ then  $a  \leq b$;
\item[(ii)] for all $a,b,c \in \NN^m$, if $a \po b$ then $a + c \po b + c$;
\item[(iii)] for all $a \in \NN^m$, we have $o \po a$, where $o=(0,\dots,0) \in \NN^m$.
\end{itemize}
Property (i) means that
$\leq$ is a linear extension of $\po$. Properties (ii) and (iii) are similar to the properties of a monomial order (see Definition \ref{defmo}). Note that  properties (ii) and (iii) together imply that $a \po a + b$ for any $a,b \in \NN^m$. From now on, we fix a pair $(\po,\leq)$ such that all three properties above are verified. The additional use of a partial order $\po$ which is compatible with the total order $\leq$ will put stronger conditions on association schemes, and hence finer structure information on the schemes will be obtained. 

We start by introducing a special type of multivariate polynomials using the property (i) of the pair $(\po,\leq)$. The following definition can be compared to Definition 2.1 in \cite{bernard2022bivariate}.
\begin{definition}\label{def:compa}
An $m$-variate polynomial $v_n(\x)$ of multidegree $n \in \NN^m$ with respect to $\leq$ is called $\po$-compatible if all monomials $\x^a$ appearing in $v_n(\x)$ satisfy $a \po n$.
\end{definition} 
Next, we refine the structure of multivariate $P$-polynomial association schemes by use of $\po$.

\begin{definition}\label{def:mvaspo}
An $m$-variate $P$-polynomial association scheme $\cZ=\{A_n \ | \ n \in \cD \}$ on the domain $\cD \subset \NN^m$ with respect to $\leq$ is called $\po$-compatible if
\begin{itemize}
\item[(ii')] for all $n\in\cD$, the $m$-variate polynomial $v_n(\x)$ of multidegree $n$ satisfying $A_n=v_n(\x)$ is $\po$-compatible;
\item[(iii')] for $i = 1,2, \dots, m$ and $a = (a_1, a_2, \dots, a_m) \in \mathcal{D}$, the product $A_{e_i}  A_{e_1}^{a_1}A_{e_2}^{a_2} \dots A_{e_m}^{a_m}$ is a
linear combination of
\begin{equation}
\{ A_{e_1}^{b_1}A_{e_2}^{b_2} \dots A_{e_m}^{b_m}\ | \ b = (b_1, \dots, b_m) \in \mathcal{D}, \ b \po a + e_i  \}.
\end{equation}
\end{itemize}
\end{definition}
Let us emphasize that Definition \ref{def:mvaspo} is an analogue of Definition \ref{defmv}, where conditions (ii) and (iii) are refined to (ii') and (iii') using the partial order $\po$, and correspond to more structured multivariate $P$-polynomial association schemes. A refined version of Proposition \ref{mprop} holds:
\begin{proposition}\label{mproppo}Let $\mathcal{D}\subset \mathbb{N}^m$ such that $e_1, e_2, \dots, e_m\in \mathcal{D}$ and $\mathcal{Z} = \{A_n \ | \ n \in \mathcal{D}\}$ be a commutative association scheme. Then the following two statements are
equivalent:
\begin{enumerate}
\item[(i')] $\mathcal{Z}$ is a $\po$-compatible $m$-variate $P$-polynomial association scheme on $\mathcal{D}$ with respect to $\leq$;
\item[(ii')] The condition (i) of Definition \ref{defmv} holds for $\mathcal{D}$ and the intersection numbers satisfy,  
for each $i = 1,2,\dots, m$ and each $a \in \mathcal{D}$, $p_{e_i, a}^b \neq 0$ for $b \in \mathcal{D}$ implies $b \po a + e_i$. Moreover, if $a + e_i \in \mathcal{D}$, then $p_{e_i, a}^{a + e_i} \neq 0$ holds.
\end{enumerate}
\end{proposition}
\begin{proof}
See Proposition \ref{propapp} in Appendix \ref{sec:appendix}. 
\end{proof}

We now define more structured multivariate distance-regular graphs by requiring certain compatibility with $\po$, and show a one-to-one correspondence with refined multivariate $P$-polynomial association schemes defined in Definition~\ref{def:mvaspo}. 
\begin{definition}\label{defabg} An $m$-distance-regular graph $G$ is called $\po$-compatible if for any two vertices $x$ and $y$, if there is a walk of $m$-length $\ell$ between them, then $d_m(x,y)\po \ell$.
\end{definition}
\begin{proposition}\label{prop:equivGraphASpo} There is a one-to-one correspondence between the $\po$-compatible $m$-distance-regular graphs and the $\po$-compatible $m$-variate $P$-polynomial association schemes.
\end{proposition}
\begin{proof}
First assume $\mathcal{Z} = \{A_n \ | \ n \in \mathcal{D}\}$ is a $\po$-compatible $m$-variate $P$-polynomial association scheme. It follows from 
Theorem \ref{realdeal} that $\mathcal{Z}$ consists of the $m$-distance matrices of the $m$-distance-regular graph $G$ with adjacency matrix $A = A_{e_1} + \cdots + A_{e_m}$. Now we show $G$ is $\po$-compatible. 
Take any two vertices $x$ and $y$ and assume there is a walk of $m$-length $\ell$ between them. 
In the case $\ell \in \cD$, making use of the following property that a $\po$-compatible $m$-variate $P$-polynomial association scheme has  (see Lemma~\ref{lem:span} in Appendix \ref{sec:appendix}): 
\begin{equation}
\text{span}\{A_a\ | \ a \in \cD, \ a \po \ell \} = \text{span}\Big\{\prod_{i = 1}^m A_{e_i}^{a_i} \ | \ a \in \cD, \ a \po \ell\Big\},  
\end{equation}
or, in the case $\ell \notin \cD$, making use of property (iii') of Definition 
\ref{def:mvaspo} and with a similar argument as in the proof of Theorem~\ref{realdeal} 
around equation \eqref{ex1}, 
we know that, in both cases, there exist coefficients $f_a$ such that
\begin{equation}\label{eqpab}
\prod_{i = 1}^m A_{e_i}^{\ell_i} = \sum_{\substack{a \in \mathcal{D} \\ a \po \ell}} f_a A_a.
\end{equation}
Combined with the fact that $\left(\prod_{i = 1}^m A_{e_i}^{\ell_i}\right)_{xy} \neq 0 $ (i.e.\ the number of walks of $m$-length $\ell$ between $x$ and $y$ is non-zero), it follows that $(A_a)_{xy} \neq 0$  for some $a \po \ell$ and thus
\begin{equation}
d_m(x,y) =a\po \ell,
\end{equation}
where the equality follows from the fact that the matrices of the scheme are the $m$-distance matrices of $G$.

Next we show that the $m$-distance matrices of a $\po$-compatible $m$-distance-regular graph form a $\po$-compatible $m$-variate $P$-polynomial association scheme. Again by Theorem \ref{realdeal} these $m$-distance matrices form an $m$-variate $P$-polynomial association scheme. It follows by Proposition \ref{mprop} that $p_{e_i,a}^{a+e_i}\neq 0$ if $a+e_i \in \cD$. Let us now assume that $p_{e_i, a}^b \neq 0 $ for some $a,b \in \cD$. This implies that there exist three vertices $x$, $y$ and $z$ such that $d_m(x,y)=b$, $d_m(x,z)=e_i$ and $d_m(y,z)=a$. Therefore there is a walk of $m$-length $a + e_i$ between $x$ and $y$. It then follows from the $\po$-compatibility of the graph that
\begin{equation}
b \po a + e_i.
\end{equation}
Therefore the association scheme consisting of the $m$-distance matrices of a $\po$-compatible $m$-distance-regular graph $G$ satisfies statement (ii') in Proposition \ref{mproppo}, and hence is a $\po$-compatible $m$-variate $P$-polynomial association scheme. 
\end{proof}

\subsection{Conditions on the domain and its boundary}
In this subsection, we define some further notions, concerning domains in $\NN^m$ and their boundary, which are necessary for comparing the definitions of bivariate $P$-polynomial schemes given in \cite{bernard2022bivariate} and \cite{bannai2023multivariate}.
 
\begin{definition}\label{def:domain}
(Definition 2.2 in \cite{bernard2022bivariate}) A subset $\mathcal{D}$ of $\mathbb{N}^m$ is called $\po$-compatible if for any $a \in \cD$, one has 
\begin{equation} 
b \po a \quad \Rightarrow  \quad  b \in \cD.
\end{equation}
\end{definition}
Note that a subset $\cD$ of $\NN^m$ which is $\po$-compatible is also said to be a downset of $(\NN^m,\po)$. 
The following result says a downset of $(\NN^m,\po)$ can be used as the domain of an $m$-variate $P$-polynomial association scheme (the $\po$-compatibility of a domain $\cD$ is in general more restraining than condition (i) of Definition \ref{defmv}). \begin{lemma}\label{lemdom} Let $\mathcal{D} \subset \mathbb{N}^m$ be a $\po$-compatible subset. Then, $\cD$ satisfies condition (i) of Definition \eqref{defmv}. 
\end{lemma}
\begin{proof}
This is a direct consequence of Definition \ref{def:domain} and the fact that $a \po a + b$ for any $a,b \in \NN^m$, as deduced from properties (ii) and (iii) of the partial order $\po$ at the beginning of Subsection \ref{ssec:porefin}.
\end{proof}
\begin{lemma}\label{polimp} Let $\mathcal{D} \subset \mathbb{N}^m$ be a $\po$-compatible subset, and let $v_n(\x)$ be a $\po$-compatible $m$-variate polynomial of multidegree $n \in \cD$. Then, all the monomials $\x^a$ in $v_n(\x)$ satisfy $a \in \cD$. 
\end{lemma}
\begin{proof}
This follows from Definitions \ref{def:compa} and \ref{def:domain}.
\end{proof}
Therefore the $\po$-compatibility of the domain $\mathcal{D}$ and of the polynomials $v_n(\x)$  naturally imply one of the requirements concerning the polynomials in condition (ii) of Definition \ref{defmv}, that is, all the monomials of $v_n(\x)$ have multidegree in $\mathcal{D}$. Moreover, if $\cZ$ is a $\po$-compatible $m$-variate $P$-polynomial association scheme on a $\po$-compatible domain $\mathcal{D}$, then the polynomials $v_n(\x)$  satisfy certain recurrence relations as explained next. 
\begin{lemma}\label{lem:corresprec}(Lemma 2.5 in \cite{bernard2022bivariate}) 
Suppose $\cZ=\{A_n \ | \ n \in \cD\}$ is an association scheme such that $\cD \subset \NN^m$ is a $\po$-compatible domain and such that $A_n = v_n(A_{e_1},\dots,A_{e_m})$ for some $\po$-compatible $m$-variate polynomials $v_n(\x)$ of multidegree $n$ with respect to $\leq$ for all $n \in \cD$. Then for any $i\in\{1,\ldots,m\}$ and $a\in \mathcal{D}$ such that $a+e_i\in\mathcal{D}$, 
  \begin{equation}
      x_i v_a(\x)=\sum_{b\po a+e_i}p_{a,e_i}^{b} v_b(\x).
  \end{equation}
\end{lemma}
\begin{proof}
    As in the proof of Lemma 2.5 in \cite{bernard2022bivariate}, it can be shown that for any $\ell \in \cD$ 
\begin{equation}\label{eq:spanpol}
\text{span}\{v_n(\x) \ | \ n \po  \ell \} = \text{span}\{\x^n \ | \ n \po \ell \}, 
\end{equation} 
and that there exist constants  $\mu_{a,e_i}^{b}$ such that 
\begin{equation}
    x_i v_a(\x)=\sum_{b\po a+e_i}\mu_{a,e_i}^{b} v_b(\x).
\end{equation} 
By evaluating the polynomials at the matrices $A_{e_1},\ldots, A_{e_m}$, we know that $A_{e_i} A_a=\sum_{b\po a+e_i}\mu_{a,e_i}^{b} A_b$, and therefore $\mu_{a,e_i}^{b}=p_{a,e_i}^{b}$.  
\end{proof}

Definition \ref{def:domain} and Lemma \ref{polimp} are essential components of the bivariate $P$-polynomial schemes of type $(\alpha,\beta)$ introduced in \cite{bernard2022bivariate}. We now introduce a different notion of compatibility involving the boundary of a domain $\cD \subset \NN^m$, which is based on (and essential to) the definition of multivariate $P$-polynomial schemes in \cite{bannai2023multivariate} and its refinement in Definition \ref{def:mvaspo}.
\begin{definition}\label{mindom}
Let $\leq$ be a monomial order on $\mathbb{N}^{m}$. An association scheme $\mathcal{Z} = \{A_n \ | \ n \in \mathcal{D} \}$ on a domain $\mathcal{D}\subset \mathbb{N}^m$ that contains $e_1, e_2, \dots,e_m$ is said to have a $\leq$-compatible boundary if for all $a \in \mathcal{D}$ such that $a + e_i \notin \mathcal{D}$ for some $i \in \{1,2,\dots,m\}$ we have
\begin{equation}\label{eq:compboundary}
A_{e_i} \prod_{j=1}^m A_{e_j}^{a_j} \ \in \ \text{span}\Big\{\prod_{j=1}^m A_{e_j}^{b_j}\ |\ b \in \mathcal{D}, \ b \leq a + e_i\Big\}.
\end{equation} 
Similarly, if $\po$ is a partial order on $\NN^m$ which satisfies properties (i)--(iii) at the beginning of Subsection \ref{ssec:porefin} with a fixed monomial order $\leq$, and if \eqref{eq:compboundary} holds with $\leq$ replaced by $\po$, we say that $\cZ$ has a $\po$-compatible boundary.
\end{definition}
Let us remark that the $\leq$-compatibility of the boundary corresponds to condition (iii) of Definition \ref{defmv}, while the $\po$-compatibility is a refinement corresponding to condition (iii') of Definition \ref{def:mvaspo}.

\subsection{Bivariate $P$-polynomial association schemes of type $(\alpha,\beta)$}

We are now ready to discuss the bivariate polynomial structures of association schemes considered in \cite{bernard2022bivariate} and provide an interpretation in terms of graphs. We restrict here to the realm of two variables since this is the case studied in \cite{bernard2022bivariate}, but all the concepts and results could naturally be generalized to the multivariate case.  

From now on, unless otherwise stated, the monomial order $\leq$ will refer to the following \textit{deg-lex} order on $\mathbb{C}[x,y]$ (or equivalently on $\mathbb{N}^2$),
\begin{equation}\label{od1}
x^{i} y^{j} \leq x^{i'} y^{j'} \iff \left\{
	\begin{array}{ll}
		 i + j < i' + j' \\
		 \text{or} \\
		i + j = i' + j' \ \text{and}\ j \leq j',
	\end{array}
\right.
\end{equation}
which is the order used in \cite{bernard2022bivariate}. Note that it is distinct from the $m=2$ case of the deg-lex order \eqref{CPord} used in \cite{bannai2023multivariate}. In \cite{bernard2022bivariate}, a specific partial order $\preceq_{(\alpha, \beta)}$ with two parameters satisfying $0\leq \alpha\leq 1$ and $0 \leq \beta < 1$ is also defined:
\begin{equation}\label{po1}
x^{i} y^{j} \preceq_{(\alpha, \beta)} x^{i'} y^{j'} \iff \left\{
	\begin{array}{ll}
		 i+ \alpha j \leq i' + \alpha j' \\
		 \text{and} \\
		\beta i + j \leq \beta i' + j'.
	\end{array}
\right.
\end{equation}
It can be checked that the partial order $\poo$ and the deg-lex total order $\leq$ satisfy properties (i)--(iii) at the beginning of Subsection \ref{ssec:porefin}. Using the concept of $\poo$-compatibility of polynomials and domains defined in the previous subsections, one can then define bivariate $P$-polynomial association schemes of type $(\alpha,\beta)$ as in \cite{bernard2022bivariate}. 

\begin{definition} \label{def:bi}(Definition 2.3. in \cite{bernard2022bivariate})
Let $\mathcal{D} \subset \mathbb{N}^2$, $0 \leq \alpha\leq 1$, $0 \leq \beta<1$ and $\preceq_{(\alpha,\beta)}$ be the partial order defined in \eqref{po1}.
 The association scheme $\mathcal{Z}$ is called bivariate $P$-polynomial of type $(\alpha,\beta)$ on the domain $\mathcal{D}$ if the following two conditions are satisfied:
 \begin{itemize}
  \item[(i)] there exists a relabeling of the adjacency matrices:
 \begin{equation}
  \{A_0,A_1,\dots, A_N\} = \{ A_{n} \ |\ n = (n_1,n_2) \in \mathcal{D} \},
 \end{equation}
such that, for $n = (n_1,n_2) \in \mathcal{D}$,
\begin{equation}
 A_{n}=v_{n_1 n_2}(A_{10},A_{01})\,, \label{eq:vij}
\end{equation}
where  $v_{n_1 n_2}(x,y)$ is a $\poo$-compatible bivariate polynomial of multidegree $(n_1,n_2)$;
\item[(ii)] $\mathcal{D}$ is $\poo$-compatible.
 \end{itemize}
\end{definition}
Bivariate $P$-polynomial association schemes of type $(\alpha,\beta)$ can be characterized in terms of their intersection numbers, by use of $\preceq_{(\alpha,\beta)}$. This correspondence is captured in the following proposition.
\begin{proposition}\label{pr:z1} (Proposition 2.6. in \cite{bernard2022bivariate}) Let $\mathcal{Z}=\{ A_{n} \ |\ n = (n_1,n_2) \in \mathcal{D} \}$ be an association scheme. The statements $(i)$ and $(ii)$ are equivalent:
\begin{itemize}
\item[(i)] $\mathcal{Z}$ is a bivariate $P$-polynomial association scheme of type $(\alpha,\beta)$ on $\mathcal{D}$;
\item[(ii)] $\mathcal{D}$ is $(\alpha,\beta)$-compatible and the intersection numbers satisfy, for  $a,b,a + e_1  \in \mathcal{D}$,
  \begin{eqnarray}
   && p_{e_1,a}^{a + e_1}\neq 0,\ \  p_{e_1,a + e_1}^{a}\neq 0, \label{eq:cm1} \\
   && p_{e_1,a}^{b}\neq 0\quad \Rightarrow \quad   b \preceq_{(\alpha,\beta)} a + e_1 ,\label{eq:cm3}
  \end{eqnarray}
  and, for  $a,b,a + e_2 \in \mathcal{D}$,
    \begin{eqnarray}
   && p_{e_2,a}^{a + e_2}\neq 0,\ \ p_{e_2, a + e_2}^{a}\neq 0,  \label{eq:cm2}\\
 && p_{e_2,a}^{b}\neq 0\quad \Rightarrow \quad   b \preceq_{(\alpha,\beta)} a + e_2. \label{eq:cm4}
  \end{eqnarray}
\end{itemize}
\end{proposition}
Note Proposition~\ref{pr:z1} is not the same as  
Proposition~\ref{mproppo} by taking the partial order $\po$ to be $\poo$, as Proposition~\ref{mproppo} corresponds to multivariate $P$-polynomial schemes satisfying the boundary condition \eqref{eq:compboundary}, and hence $a+e_i$ is not required to be in $\mathcal{D}$ for the implication $(p_{e_i,a}^{b}\neq 0\; \Rightarrow \;   b \po a + e_i)$ to hold. 

We now emphasize two differences between Definition \ref{def:bi} and Definitions \ref{defmv} and its refinement \ref{def:mvaspo}. First, the $\poo$-compatibility of the domain in Definition \ref{def:bi} implies the condition (i) of Definition \ref{defmv}, but the converse is not true. Second, Definition \ref{def:bi} does not require $\leq$-compatibility or $\poo$-compatibility of the boundary; this constrasts with Definition \ref{def:mvaspo}. Therefore Definition \ref{def:bi} is more restrictive on the domain, but less restrictive on the Bose-Mesner relations at the boundary. 

As a special case of Lemma \ref{lem:corresprec}, the matrices and the bivariate polynomials $v_{n_1n_2}$ of a bivariate $P$-polynomial association scheme of type $(\alpha,\beta)$ satisfy the same recurrence relations, as long as all the polynomial multidegrees are in the domain $\mathcal{D}$. 
That is, for $\mathcal{D}\subset \mathbb{N}^2$, $a \in \cD$ and $i=1,2$, if $a+e_i\in\mathcal{D}$, then 
\begin{equation} \label{eq:recAs}
A_{e_i }A_{a} =\sum_{ \ell \preceq_{(\alpha,\beta)} a+e_i} p_{e_i,a}^{\ell}A_\ell
\end{equation} 
and 
\begin{equation}
x_i v_a =\sum_{ \ell \preceq_{(\alpha,\beta)} a+e_i} p_{e_i,a}^{\ell} v_\ell, 
\end{equation} 
where we denote $x_1=x$ and $x_2=y$. In contrast, the polynomials associated to a multivariate $P$-polynomial association scheme in the sense of Definition~\ref{defmv} may not satisfy the same recurrences as the matrices do. In fact, the multidegrees of some monomials in $x_iv_a$ might not belong to the domain $\mathcal{D}$. 

On the other hand, the restriction on the indices in \eqref{eq:recAs} may not hold if $a+e_i\notin\mathcal{D}$.
This differs from $\poo$-compatible bivariate $P$-polynomial association schemes in the sense of Definitions~\ref{defmv} and \ref{def:mvaspo} (by taking $m=2$ and $\po$ to be $\poo$), in which case \eqref{eq:recAs} holds even for $a+e_i\notin\mathcal{D}$ because of Proposition~\ref{mproppo}. Also note that without the $\poo$-compatibility, Proposition~\ref{mprop} gives an analogue of \eqref{eq:recAs}:
\begin{equation} \label{eq:recAs2}
A_{e_i}A_{a} =\sum_{ \ell \leq a+e_i} p_{e_i,a}^{\ell}A_\ell.
\end{equation}
This difference highlights that, 
given a bivariate $P$-polynomial association scheme of type $(\alpha,\beta)$ on a domain $\mathcal{D}$, there is no guarantee that it is also bivariate $P$-polynomial in the sense of Definitions \ref{defmv} and/or \ref{def:mvaspo} on the same domain $\mathcal{D}$ and with respect to the deg-lex order \eqref{od1} and partial order $\poo$, as it is not true that the former schemes always have a compatible boundary. This will be illustrated in the example of the generalized $24$-cell schemes in Subsection~\ref{ssec:24cell}. Likewise, various examples in \cite{bannai2023multivariate} made it clear that there exist bivariate $P$-polynomial schemes in the sense of Definition \ref{defmv} that are not of type $(\alpha,\beta)$, as the domain might fail to be $\poo$-compatible. Note however that most known examples of bivariate $P$-polynomial schemes of type $(\alpha,\beta)$ have a $\leq$-compatible boundary for some order $\leq$. This is the case for all the examples discussed in \cite{bernard2022bivariate} as pointed out in Remark 4 of \cite{bannai2023multivariate}. In fact, since most of the examples in \cite{bernard2022bivariate} are studied via the intersection numbers $p_{e_1,a}^{b}$ and $p_{e_2,a}^{b}$ for any $a=(a_1,a_2)$ and $b=(b_1,b_2)$ in the domain $\cD \subset \NN^2$ (including the boundary), condition (ii') of Proposition \ref{mproppo} can be used to deduce that these association schemes have a $\poo$-compatible boundary. 

It follows from this discussion that none of the two definitions of bivariate $P$-polynomial schemes implies the other. Yet, we know from many examples that they overlap. The cases where this overlap happens is characterized in the following propositions, which can both be proved straightforwardly by use of Definitions \ref{defmv} and \ref{def:mvaspo} and Lemmas \ref{lemdom} and \ref{polimp}.  

\begin{proposition}\label{rel2def} Let $\mathcal{Z}$ be a bivariate $P$-polynomial association scheme of type $(\alpha,\beta)$ on a domain $\mathcal{D}$, in the sense of Definition \ref{def:bi}. Then it is also a $\poo$-compatible bivariate $P$-polynomial association scheme on the domain $\mathcal{D}$ in the sense of Definitions \ref{defmv} and \ref{def:mvaspo} with the deg-lex monomial order $\leq$ defined in \eqref{od1} and the partial order $\poo$ defined in \eqref{po1} if and only if it has a $\poo$-compatible boundary. 
\end{proposition}

\begin{proposition}\label{rel2def2} Let $\mathcal{Z}$ be a $\poo$-compatible bivariate $P$-polynomial association scheme on the domain $\mathcal{D}$ in the sense of Definitions \ref{defmv} and \ref{def:mvaspo} with the deg-lex monomial order $\leq$ defined in \eqref{od1} and the partial order $\poo$ defined in \eqref{po1}. Then it is also a bivariate $P$-polynomial association scheme of type $(\alpha,\beta)$ in the sense of Definition \ref{def:bi} if and only if it has a $\poo$-compatible domain $\mathcal{D}$. 
\end{proposition}
Multivariate $P$-polynomial schemes in Definition \ref{defmv} are in one-to-one correspondence with multivariate distance regular graphs. Not all bivariate $P$-polynomial association schemes of type $(\alpha,\beta)$ give rise to a 2-distance regular graph. 
For the ones with a compatible boundary we have the following result. 
\begin{proposition} There is a one-to-one correspondence between the $\poo$-compatible $2$-distance-regular graphs with a $\poo$-compatible domain $\cD\subset \NN^2$ and the bivariate $P$-polynomial association schemes of type $(\alpha,\beta)$ with a $\poo$-compatible boundary.
\end{proposition}
\begin{proof}
This directly follows from Propositions \ref{prop:equivGraphASpo}, \ref{rel2def} and \ref{rel2def2}.
\end{proof}

Let us stress again that bivariate $P$-polynomial schemes of type $(\alpha,\beta)$ are associated to polynomials satisfying a special type of recurrence relations on $\mathbb{C}[x,y]$ \cite{bernard2022bivariate}. The results in this section can be seen as connecting a subset of $2$-distance-regular graphs with the investigation of polynomials in $\mathbb{C}[x,y]$ with special properties, further motivating the study of bivariate $P$-polynomial association schemes of type $(\alpha,\beta)$ with a $\poo$-compatible boundary. 
For a general $m\in\mathbb{N}$, this corresponds to $\po$-compatible $P$-polynomial association schemes on a $\po$-compatible domain, which have a graph interpretation and  the associated polynomials satisfy certain recurrences on $\mathbb{C}[\x]$. 

\subsection{Generalized 24-cell}\label{ssec:24cell}

In \cite{MMW} (see Theorem 3.6) a family of association schemes with 4 classes which generalize the $24$-cell is constructed and shown to be  $Q$-polynomial but not $P$-polynomial. Recently, they were shown to be bivariate $P$-polynomial association schemes of type $(0,0)$ \cite{bernard2022bivariate}. We call such an association scheme generalized 24-cell. Here we highlight the fact that they are bivariate $P$-polynomial of type $(\alpha,\beta)$ with respect to multiple domains and that, according to the choice of domain, the scheme may or may not be bivariate $P$-polynomial in the sense discussed in Section \ref{s2} with respect to the deg-lex monomial order \eqref{od1}.

The matrices $L_i$ of intersection numbers of a generalized 24-cell association scheme $\mathcal{Z} = \{A_0, A_1, A_2, A_3, A_4\}$, with entries $(L_i)_{kj}=p_{ij}^k$, are given by $L_0=I_5$, 
\begin{eqnarray}
 &&L_1=\begin{pmatrix}
        0 & 16\ell s^2 & 0 & 0 & 0 \\
        1 & 2(\ell-1)s(4s+1) & (4s-1)(4s+1)& 2(\ell-1)s(4s-1)& 0\\
        0 & 8\ell s^2 & 0 & 8\ell s^2 & 0\\
         0 & 2(\ell-1)s(4s-1) & (4s-1)(4s+1)& 2(\ell-1)s(4s+1)& 1\\
         0 & 0 & 0 & 16\ell s^2 & 0
       \end{pmatrix},\\
        &&  L_2=\begin{pmatrix}
            0 & 0  & 2(4s-1)(4s+1) & 0 & 0\\
            0 & (4s-1)(4s+1) & 0 & (4s-1)(4s+1) & 0\\
            1 & 0 & 32s^2-4 & 0 & 1 \\
               0 & (4s-1)(4s+1) & 0 & (4s-1)(4s+1) & 0\\
               0 & 0  & 2(4s-1)(4s+1) & 0 & 0
           \end{pmatrix},\\
        && L_3=\begin{pmatrix}
        0 & 0 & 0 & 16\ell s^2 & 0 \\
        0 & 2(\ell-1)s(4s-1) & (4s-1)(4s+1)& 2(\ell-1)s(4s+1)& 1\\
        0 & 8\ell s^2 & 0 & 8\ell s^2 & 0\\
         1 & 2(\ell-1)s(4s+1) & (4s-1)(4s+1)& 2(\ell-1)s(4s-1)& 0\\
         0 & 16\ell s^2 & 0 & 0 &  0
       \end{pmatrix},
\end{eqnarray}
and
\begin{eqnarray}
       &&L_4=\begin{pmatrix}
              0 & 0 & 0 & 0 & 1\\
               0 & 0 & 0 & 1 & 0\\
                0 & 0 & 1 & 0 & 0\\
                 0 & 1 & 0 & 0 & 0\\
                  1 & 0 & 0 & 0 & 0
             \end{pmatrix}.
\end{eqnarray}

The parameters $\ell$ and $s$ characterize the different generalizations. The 24-cell association scheme in Section \ref{subse:24cell} corresponds to the case $\ell=2$ and $s=\frac12$. The matrices $A_0$, $A_1$, $A_2$, $A_3$ and $A_4$ are then associated to the Euclidean distances in $\mathbb{R}^4$ of 0, $\sqrt{2}$, 2, $\sqrt{6}$ and $\sqrt{8}$, respectively. 

Now, consider the two subsets of $\mathbb{N}^2$ illustrated in Figure~\ref{Fig:twoDsfor24cell} and defined as
\begin{equation}
 \mathcal{D}_1 = \{(0,0), (1,0), (0,1), (2,0),(1,1)\},\quad \mathcal{D}_2 = \{(0,0), (1,0), (0,1), (2,0), (0,2)\}.
\end{equation}

\begin{figure}[htbp]
\begin{center}
\begin{tikzpicture}[scale=1.0]
\draw[->] (0,0)--(4,0);\draw[->] (0,0)--(0,3);
\draw [fill] (0,0) circle (0.07);
\draw [fill] (1,0) circle (0.07);
\draw [fill] (2,0) circle (0.07);
\draw [fill] (0,1) circle (0.07);
\draw [fill] (1,1) circle (0.07);
\draw (0,3) node[above] {$n_2$};
\draw (4,0) node[above] {$n_1$};
\draw (2,-.5) node[below] {$\mathcal{D}_1$};
\end{tikzpicture} 
\hspace{3cm}
\begin{tikzpicture}[scale=1.0]
\draw[->] (0,0)--(4,0);\draw[->] (0,0)--(0,3);
\draw [fill] (0,0) circle (0.07);
\draw [fill] (1,0) circle (0.07);
\draw [fill] (2,0) circle (0.07);
\draw [fill] (0,1) circle (0.07);
\draw [fill] (0,2) circle (0.07);
\draw (0,3) node[above] {$n_2$};
\draw (4,0) node[above] {$n_1$};
\draw (2,-.5) node[below] {$\mathcal{D}_2$};
\end{tikzpicture}
\end{center}
\caption{Two domains on which the generalized 24-cell association schemes are bivariate $P$-polynomial of type $(\al,\be)$, with $0\leq \al \leq 1$, $0\leq \be < 1$ in the first case and $\frac{1}{2} \leq \al < 1$, $0\leq \be < 1$ in the second case.}
\label{Fig:twoDsfor24cell}
\end{figure}
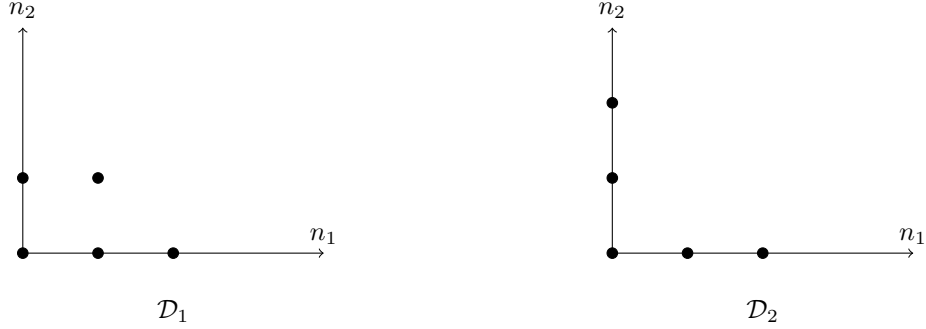

For any $0\leq\alpha\leq 1$, and $0\leq\beta<1$, 
the subset $\mathcal{D}_1$  is $\poo$-compatible and so are the  following two bivariate polynomials
\begin{equation}
v_{11}(x,y)=\frac{1}{(4s-1)(4s+1)}xy-y,\quad 
v_{20}(x,y)=\frac{1}{2(4s-1)(4s+1)}x^2-\frac{2(8s^2-1)}{(4s-1)(4s+1)} x-1.
\end{equation}
It follows that the generalized 24-cell association schemes $\mathcal{Z}$ are bivariate $P$-polynomial of type $(\alpha,\beta)$ on $\mathcal{D}_1$ with the labelling
\begin{equation}\label{eq:AD1}
 A_{00}=A_0\ , \quad A_{10}=A_2\ , \quad A_{01}=A_3\ , \quad A_{11}=A_1\ , \quad A_{20}=A_4.
\end{equation}
From the intersection numbers in the matrices $L_i$, one indeed finds that 
\begin{equation} \label{eq:poly11}
 A_{11}=v_{11}(A_{10},A_{01}), \quad A_{20}=v_{20}(A_{10},A_{01}),
\end{equation}
so that both conditions of Definition \ref{def:bi} are verified.
\begin{remark} There can exist more than one  labellings of the matrices $\{A_n \, | \, n\in\mathcal{D}\}$ so that a scheme is bivariate $P$-polynomial of type $(\alpha,\beta)$ on a given domain $\mathcal{D}$. For instance, generalized $24$-cell association schemes are also bivariate $P$-polynomial of type $(\alpha,\beta)$ on $\mathcal{D}_1$ with the polynomials $v_{ij}$ (as above) and the labelling $A_{00}=A_0$, $A_{10}=A_2$, $A_{01}=A_1$, $A_{11}=A_3$, $A_{20}=A_4$. 
\end{remark}

For $0\leq\alpha <1$ and $0\leq\beta <1$, the subset $\mathcal{D}_2$ is also $(\alpha,\beta)$-compatible.  If we further have $\frac12\leq\alpha<1$, 
then the same holds for the bivariate polynomial
\begin{equation}
v_{02}(x,y)=\frac{1}{2(\ell-1)s(4s+1)}\Big(y^2-2(\ell-1)s(4s-1)y-8\ell s^2x-16\ell s^2\Big).
\end{equation}
It follows that the generalized 24-cell association schemes $\mathcal{Z}$ are also bivariate $P$-polynomial of type $(\alpha,\beta)$ on $\mathcal{D}_2$ with $\frac12\leq\alpha<1$, $0\leq\beta<1$ and with the labelling 
\begin{equation}\label{eq:AD2}
 A_{00}=A_0\ , \quad A_{10}=A_2\ , \quad A_{01}=A_3\ , \quad A_{02}=A_1\ , \quad A_{20}=A_4.
\end{equation}
Indeed, one can check using the intersection numbers in $L_i$ that
\begin{equation} \label{eq:poly11}
 A_{02}=v_{02}(A_{10},A_{01}), \quad A_{20}=v_{20}(A_{10},A_{01}).
\end{equation} 
The two labellings show that there can be more than one region $\mathcal{D}$ on which an association scheme is bivariate $P$-polynomial of type $(\alpha,\beta)$.

We now check whether the generalized $24$-cell schemes are also bivariate $P$-polynomials in the sense of Definition \ref{defmv} with respect to the monomial order in \eqref{od1} with the labellings \eqref{eq:AD1}
and \eqref{eq:AD2} associated to the domains $\mathcal{D}_1$ and $\mathcal{D}_2$, respectively. For both labellings, it is straightforward to check that conditions (i) and (ii) of this definition are verified. Condition (iii) also holds with the labelling \eqref{eq:AD1} and domain $\mathcal{D}_1$. Furthermore, one can verify that the refined conditions (ii') and (iii') of Definition \ref{def:mvaspo} hold with the partial order $\poo$ if $\alpha=1$ and $0\leq \beta < 1$.  However, we have the following relation in terms of the labelling \eqref{eq:AD2}, 
\begin{equation}
A_{10}A_{01}=(4s-1)(4s+1)A_{02}+(4s-1)(4s+1)A_{01}. 
\end{equation}
Since $(0,2) \nleq (1,1)$ with the order \eqref{od1}, it follows that condition (iii) of Definition \ref{defmv} is not satisfied with the usual monomial order \eqref{od1} and the labelling \eqref{eq:AD2} associated with the domain $\mathcal{D}_2$. Hence, the generalized $24$-cell schemes $\mathcal{Z}$,
which are bivariate $P$-polynomial schemes of type $(\alpha,\beta)$ on both $\mathcal{D}_1$ and $\mathcal{D}_2$, 
are bivariate $P$-polynomial on $\mathcal{D}_1$ in the sense of Definition \ref{defmv} with respect to the monomial order \eqref{od1} (and even in the more restrictive case of Definition \ref{def:mvaspo} using the partial order $\preceq_{(1,\beta)}$ for any $0\leq \beta<1$), but not on $\mathcal{D}_2$. 

\begin{remark} If one uses instead the case $m=2$ of the monomial order $\leq$ defined in \eqref{CPord}, the opposite situation arises: $\mathcal{Z}$ are bivariate $P$-polynomial schemes in the sense of Definition \ref{defmv} on $\mathcal{D}_2$, but not on $\mathcal{D}_1$. This order is usually not considered in the framework of $(\alpha, \beta)$-compatible schemes since it is only implied by the partial order $\preceq_{(\alpha,\beta)}$ when $\alpha \neq 1$.
\end{remark}

\section{Outlook}\label{s6}

This paper introduced and explored the concept of $m$-distance-regular graphs, establishing a significant correspondence between these graphs and the newly introduced $m$-variate $P$-polynomial association schemes. Refined notions and additional constraints for these association schemes and graphs using a partial order with which the original monomial order is compatible were also examined. In particular, a correspondence between a class of $2$-distance-regular graphs and a class of bivariate $P$-polynomial schemes of type $(\alpha,\beta)$ was established. As an outlook, we propose a list of questions for further investigation.

\begin{enumerate}
\item Are there regularity conditions characterizing the set of $m$-distance-regular graphs with $m \in \mathbb{N}$?

\item What is the relation between a graph being $m$-distance-regular and its set of automorphisms? Can a notion of $m$-distance-transitivity be defined which would imply $m$-distance-regularity?

\item For a given $k = (k_1, k_2,\dots, k_m)$, are there finitely many $m$-distance-regular graph $G=(X,\Gamma_1\sqcup\Gamma_2\sqcup \cdots \sqcup \Gamma_m)$ such that the graph $(X,\Gamma_i)$ is  $k_i$ regular?

\item What properties do distance-regular and $m$-distance-regular graphs share? What can we say about the spectral properties of $m$-distance-regular graphs?

\item Which graphs are associated to the other known examples of multivariate $P$-polynomial schemes? Which $m$-distance-regular graphs are associated to the generalized Johnson schemes? What about schemes based on attenuated spaces? 

\item Under what conditions can a graph be both $m$-distance-regular and $n$-distance-regular with $n \neq m$?

\item Is there an association scheme that is bivariate $P$-polynomial of type $(\alpha,\beta)$, but is not bivariate $P$-polynomial in the sense of Definition \ref{defmv}, no matter what labelling or what monomial order we use? What about the inverse statement? 
\end{enumerate}

\section*{Acknowledgement}
We thank E. Bannai and H. Kurihara for the inspiring discussions and for sharing with us a draft of \cite{bannai2023multivariate} before its publication. PAB and MZ hold an Alexander-Graham-Bell scholarship from the Natural Sciences and Engineering Research Council of Canada (NSERC). NC is partially supported by Agence Nationale de la Recherche Projet AHA ANR-18-CE40-0001 and by the IRP AAPT of CNRS. The research of LV is supported by a Discovery Grant from NSERC.

\appendix

\section{Complementary proofs} \label{sec:appendix}

This appendix presents the proofs of Propositions \ref{mprop} and \ref{mproppo}, starting with a useful lemma. The necessary notions are defined in Section \ref{s2} and Subsection \ref{ssec:porefin}. We will use the notations $\bA=(A_{e_1},A_{e_2},\dots,A_{e_m})$ and $\bA^a=\prod_{i=1}^m A_{e_i}^{a_i}$ for $a=(a_1,a_2,\dots,a_m) \in \NN^m$.

The following result is based on part of the proof of Lemma 2.14 in \cite{bannai2023multivariate}, but slightly generalized (since $b$ in equation \eqref{eq:span} below is not required to be in $\cD$).

\begin{lemma}\label{lem:span}
Let $\cD \subset \NN^m$ such that $e_1,e_2,\dots,e_m \in \cD$, let $\leq$ be a monomial order on $\NN^m$ and let $\cZ=\{A_n \ | \ n \in \cD \}$ be an association scheme. 
Take any $b \in \NN^m$. 
If for all $n \in \cD$ such that $n \leq b$ we have $A_n = v_n(\bA)$ for some $m$-variate polynomial $v_n$ of multidegree $n$, all of whose monomials are in $\cD$, then 
\begin{equation}
\lspan\{A_n \ | \ n \in \cD, \ n \leq b \} = \lspan\{\bA^n \ | \ n \in \cD, \ n \leq b \}. \label{eq:span}
\end{equation}
Furthermore, if $\po$ is a partial order on $\NN^m$ satisfying the properties \textnormal{(i)--(iii)} at the beginning of Subsection \ref{ssec:porefin}, and if the above polynomials $v_n$ are $\po$-compatible, then \eqref{eq:span} holds with $\leq$ replaced by $\po$.
\end{lemma}
\begin{proof}
Suppose that $n \in \cD$ and $n \leq b$. We have by hypothesis that $A_n$ can be expressed as a linear combination of elements $\bA^a$ with $a \in \cD$ and $a \leq n$. By transitivity of the total order $\leq$, we  also have $a \leq b$. This shows that
\begin{equation}
\lspan\{A_n \ | \ n \in \cD, \ n \leq b \} \subseteq \lspan\{\bA^n \ | \ n \in \cD, \ n \leq b \}.\label{eq:incspan}
\end{equation}
It follows that
\begin{align}
\dim(\lspan\{A_n \ | \ n \in \cD, \ n \leq b \}) \leq \dim(\lspan\{\bA^n \ | \ n \in \cD, \ n \leq b \}) \leq |\{ n \in \cD, \ n \leq b \}|.
\end{align}
Because the adjacency matrices $\{A_n \ | \ n \in \cD \}$ are linearly independent, we have 
\begin{align}
|\{ n \in \cD, \ n \leq b \}| &= \dim(\lspan\{A_n \ | \ n \in \cD, \ n \leq b \}).
\end{align}
Therefore the dimensions of both sides of \eqref{eq:incspan} are equal, and \eqref{eq:span} holds. 

Now assume the polynomials $v_n$ are $\po$-compatible for all $n \in \cD$ with $n \po b$.  Take any such $n$, then by Definition \ref{def:compa}, $A_n$ can be expressed as a linear combination of elements $\bA^a$ with $a \in \cD$ and $a \po n$. By transitivity of $\po$, we have $a \po b$. The rest of the proof follows the same lines as previously, with the monomial order $\leq$ on $\NN^m$ replaced by $\po$.
\end{proof}

Condition (iii) of Definition \ref{defmv} is about monomials of multidegree adjacent to the boundary of $\cD$ but not in $\cD$. 
Let us introduce the notion of extended domain $\cDext$, which consists of the domain $\cD$ together with the elements in $\NN^m$ that are just outside of $\cD$, near the boundary:
\begin{equation}
\cDext  = \cD \cup \{ a \in \NN^m \ | \ \exists \ i \ \text{s.t.} \ a-e_i \in \cD \}. 
\end{equation}
\begin{lemma}\label{statement} 
Let $\mathcal{D}\subset \mathbb{N}^m$ such that $e_1, e_2, \dots, e_m\in \mathcal{D}$ and $\mathcal{Z} = \{A_n \ | \ n \in \mathcal{D}\}$ be a commutative association scheme. Suppose 
 condition (i) of Definition \ref{defmv} holds for $\mathcal{D}$. Suppose the  intersection numbers satisfy,  
for each $i = 1,2,\dots, m$ and each $a \in \mathcal{D}$, $p_{e_i, a}^b \neq 0$ for $b \in \mathcal{D}$ implies $b \leq a + e_i$, with $p_{e_i, a}^{a + e_i} \neq 0$ if  $a + e_i \in \mathcal{D}$. 
Then for $a \in \cDext$, we have 
\begin{itemize}
\item[(i)] $\bA^a \in \lspan\{\bA^b \ | \ b \in \cD, b \leq a\}$, and furthermore
\item[(ii)] if $a \in \cD$, then $A_a = v_a(\bA)$ for some $m$-variate polynomial $v_a$  of multidegree $a$, all of whose monomials are in $\cD$.
\end{itemize}
If further 
the conditions on the intersection numbers in the statement still hold when 
a total order $\leq$  is replaced by the partial order $\po$, which satisfies the properties \textnormal{(i)--(iii)} at the beginning of Subsection \ref{ssec:porefin}, then   (i) still holds when $\leq$ is  replaced by $\po$, and  (ii) still holds when $v_a$ is additionally required to be  $\po$-compatible. 
\end{lemma} 
\begin{proof}
We prove the result by induction on $\leq$. 
The domain $\cD \subset \cDext$ contains $o=(0,\dots,0)\in \NN^m$, with $A_{o}$ the identity matrix. The statement is therefore true for $a=o$, with $v_o=1$. Now we suppose that the statement is true for all $b \in \cDext$ such that $b < a$ for some fixed $a \in \cDext$, $a>o$. We want to show that it is true for $a$. There are two possible cases.

{Case 1:} suppose $a \in \cD$. In this case, part (i) is trivial. We want to prove part (ii). 
Since $a>o$, we have $a_i\geq 1$ for some $i$.  By condition (i) of Definition \ref{defmv},   $a-e_i \in \cD$ holds.  By assumption, we have 
\begin{equation}
A_{e_i}A_{a-e_i} = p_{e_i,a-e_i}^{a} A_{a} + \sum_{\substack{b \in \cD \\ b < a }} p_{e_i,a-e_i}^{b} A_{b}, \quad \text{with} \ p_{e_i,a-e_i}^{a}  \neq 0. \label{eq:prinduc5}
\end{equation} 
For $b \in \cD$ such that $b < a$, we can use part (ii) of the induction hypothesis to write 
\begin{equation}
\sum_{\substack{b \in \cD \\ b < a }} p_{e_i,a-e_i}^{b} A_{b} = \sum_{\substack{b \in \cD \\ b < a}} p_{e_i,a-e_i}^{b}  \sum_{\substack{c \in \cD \\ c \leq  b}} {g}_{b,c} \bA^{c} = \sum_{\substack{c \in \cD \\ c <  a}} {g'}_{c} \bA^{c}. \label{eq:prinduc8}
\end{equation}
Since $a-e_i<a$, by part (ii) of the induction hypothesis we have $A_{a-e_i}=v_{a-e_i}(\bA)=\sum_{\substack{c \in \cD \\ c \leq  a-e_i}} {f}_{c} \bA^{c}$ with $f_{a-e_i}\neq 0$ and hence 
\begin{equation}
A_{e_i}A_{a-e_i} = 
{f}_{a-e_i} \bA^{a} + \sum_{\substack{c \in \cD \\ c <  a-e_i}} {f}_{c} \bA^{c+e_i}, \quad \text{with} \ f_{a-e_i}  \neq 0 . \label{eq:prinduc6}
\end{equation}
The terms in the sum of equation \eqref{eq:prinduc6} satisfy $c +e_i < a$ and $c+e_i \in \cDext$ since $c \in \cD$.  We can therefore apply part (i) of the induction hypothesis to each of these terms to deduce that
\begin{equation}
A_{e_i}A_{a-e_i} = {f}_{a-e_i} \bA^{a} + \sum_{\substack{c \in \cD \\ c <  a}} {f'}_{c} \bA^{c}, \quad \text{with} \ f_{a-e_i}  \neq 0 . \label{eq:prinduc7}
\end{equation}
Combining the results \eqref{eq:prinduc7} and \eqref{eq:prinduc8} in \eqref{eq:prinduc5}, we can express $A_{a}$ as a linear combination of $\bA^c$ with $c \in \cD$ and $c \leq a$, where the   coefficient for $\bA^a$ is non-zero. In other words, $A_a=v_a(\bA)$ for some  $m$-variate polynomial  $v_a$ of multidegree $a$ such that all the  monomials in $v_a(\x)$ are in $\cD$. This proves the induction statement in this case.\\

{Case 2:} suppose $a \notin \cD$. Then we only need to prove part (i) of the claim. By definition of $\cDext$ there exists an $i$ such that $a-e_i \in \cD$. By part (ii) of the induction hypothesis and Lemma \ref{lem:span}, equation \eqref{eq:span} holds for any $b \leq a$. In particular, it holds for $b=a-e_i < a$. Therefore we have 
\begin{equation}
\bA^a=A_{e_i}\bA^{a-e_i} = \sum_{\substack{c \in \cD \\ c \leq a-e_i}} {f}_{c} A_{e_i}A_{c} = \sum_{\substack{c \in \cD \\ c \leq a-e_i}} {f}_{c} \sum_{\substack{\ell \in \cD \\ \ell \leq c+e_i}} p_{e_i,c}^{\ell}A_{\ell} =  \sum_{\substack{\ell \in \cD \\ \ell < a}}{f'}_{\ell}A_{\ell},
\end{equation}
where we used the assumption on the intersection numbers in the third equality, and in the fourth equality $\ell$ cannot be equal to $a$ because $a \notin \cD$. Now since $\ell<a$, we can use part (ii) of the induction hypothesis to write $A_\ell$ as a linear combination of elements $\bA^c$ with $c \in \cD$ and $c \leq \ell < a$. This shows that the induction statement holds in this case.

By induction, we conclude the result  holds for all $a \in \cDext$.

Suppose now that the conditions on the intersection numbers still hold when 
the total order $\leq$  is replaced by a partial order $\po$ satisfying the given conditions.  
The proof of the stronger consequence in the lemma can still be done by induction on the total order $\leq$ and follows the same lines as above. One must replace $\leq$ by $\po$, except when using the induction hypothesis. The transitivity of the partial order $\po$ and the properties (i) and (ii) at the beginning of Subsection \ref{ssec:porefin} are needed in the proof. 
\end{proof}

The next result corresponds to Proposition 2.15 of \cite{bannai2023multivariate}.  
\begin{proposition}\label{propapp} Let $\mathcal{D}\subset \mathbb{N}^m$ such that $e_1, e_2, \dots, e_m\in \mathcal{D}$ and $\mathcal{Z} = \{A_n \ | \ n \in \mathcal{D}\}$ be a commutative association scheme. Then the following two statements are
equivalent:
\begin{enumerate}
\item[(i)] $\mathcal{Z}$ is an $m$-variate $P$-polynomial association scheme on $\mathcal{D}$ with respect to a monomial order $\leq$;
\item[(ii)] The condition (i) of Definition \ref{defmv} holds for $\mathcal{D}$ and the intersection numbers satisfy,  
for each $i = 1,2,\dots, m$ and each $a \in \mathcal{D}$, $p_{e_i, a}^b \neq 0$ for $b \in \mathcal{D}$ implies $b \leq a + e_i$. Moreover, if $a + e_i \in \mathcal{D}$, then $p_{e_i, a}^{a + e_i} \neq 0$ holds.
\end{enumerate}
Suppose furthermore that $\po$ is a partial order on $\NN^m$ satisfying the properties \textnormal{(i)--(iii)} at the beginning of Subsection \ref{ssec:porefin}. Then statement (i) with the additional assumption that $\cZ$ is $\po$-compatible is equivalent to statement (ii) with $\leq$ replaced by $\po$.
\end{proposition}
\begin{proof}
(i)~$\Rightarrow$~(ii). We need to show that, for any $a \in \cD$ and any $i\in \{1,2,\dots,m\}$,
\begin{align}
A_{e_i}A_{a} = \sum_{\substack{b \in \cD \\ b \leq  a + e_i}} p_{e_i,a}^{b}A_{b}, \qquad \text{with} \ p_{e_i,a}^{a + e_i} \neq 0 \ \text{if} \ a+e_i \in \cD. \label{eq:Ametric}
\end{align}
By condition (ii) of Definition \ref{defmv} we have
\begin{equation}
A_{e_i}A_{a} = A_{e_i}v_{a}(\bA) = \sum_{\substack{b \in \cD \\ b \leq  a}} f_b \bA^{b+e_i}, \qquad \text{with} \  f_{a}\neq 0.
\end{equation}
By condition (iii) of Definition \ref{defmv}, since $b \in \cD$ we have that $\bA^{b+e_i}$ can be expressed as a linear combination of elements $\bA^c$ with $c \in \cD$ and $c \leq b + e_i \leq a + e_i$. Therefore, we can write
\begin{equation}
A_{e_i}A_{a}  = \sum_{\substack{c \in \cD \\ c \leq  a+e_i}} f'_c \bA^{c}.
\end{equation}
Using Lemma \ref{lem:span}, we deduce that \eqref{eq:span} holds for all $b \in \NN^m$ (in particular for $b=a+e_i$), hence
\begin{equation}
A_{e_i}A_{a}  = \sum_{\substack{c \in \cD \\ c \leq  a+e_i}} f''_c A_{c}. \label{eq:sumA10}
\end{equation}
If $a+e_i \in \cD$, then $f'_{a+e_i} = f_{a} \neq 0$. Moreover, by Lemma \ref{lem:span}, we also deduce that $f''_{a+e_i} \neq 0$. Since the adjacency matrices form a basis of the Bose--Mesner algebra, the coefficients appearing in \eqref{eq:sumA10} correspond to the intersection numbers, that is $f''_{c}=p_{e_i,a}^{c}$. This shows \eqref{eq:Ametric}. 

If further, $\cZ$ is $\po$-compatible, then the proof follows the same lines, with $\leq$ replaced by $\po$. One must use: Definition \ref{def:compa}, conditions (ii') and (iii') of Definition \ref{def:mvaspo} instead of conditions (ii) and (iii) of Definition \ref{defmv}, the transitivity of the partial order $\po$, and the property (ii) at the beginning of Subsection \ref{ssec:porefin} with $c=e_i$. 

(ii)~$\Rightarrow$~(i). 
By Lemma~\ref{statement}, 
  conditions (ii) and (iii) of Definition \ref{defmv} hold, in addition to condition (i), therefore $\mathcal{Z}$ is an $m$-variate $P$-polynomial association scheme. The $\po$-compatibility of $\mathcal{Z}$ under the stronger conditions follows similarly. 
\end{proof}

\bibliographystyle{IEEEtran}
\bibliography{IEEEabrv,mybibfile} 

\end{document}